\documentclass[12pt, regno]{amsart}

\usepackage{amssymb}
\usepackage{amscd}
\usepackage{amsfonts}
\usepackage{setspace}
\usepackage{version}
\usepackage{endnotes}
\usepackage{float}
\usepackage{color}

\newtheorem{theorem}{Theorem}[section]
\newtheorem{lemma}[theorem]{Lemma}
\newtheorem{proposition}[theorem]{Proposition}

\newtheorem{corollary}[theorem]{Corollary}

\theoremstyle{remark}

\theoremstyle{definition}

\theoremstyle{remark}

\numberwithin{equation}{section}

\newcommand{\Z}{\mathbb{Z}}
\newcommand{\R}{\mathbb{R}}

\newcommand{\C}{\mathbb{C}}
\newcommand{\E}{\mathbb{E}}

\newcommand{\e}{{\rm e}}

\newcommand{\ee}{\varepsilon}
\newcommand{\1}{\mathbf{1}}

\newcommand{\dd}{{\rm d}}

\begin{document}

\title[Bombieri-Vinogradov for multiplicative functions]{Bombieri-Vinogradov for multiplicative functions, and beyond the $x^{1/2}$-barrier}

\author[A. Granville]{Andrew Granville}
\address{AG: D\'epartement de math\'ematiques et de statistique\\
Universit\'e de Montr\'eal\\
CP 6128 succ. Centre-Ville\\
Montr\'eal, QC H3C 3J7\\
Canada; 
and Department of Mathematics \\
University College London \\
Gower Street \\
London WC1E 6BT \\
England.
}
\thanks{A.G.~has received funding from the
European Research Council  grant agreement n$^{\text{o}}$ 670239, and from NSERC Canada under the CRC program.}
\email{{\tt andrew@dms.umontreal.ca}}
\author[X. Shao]{Xuancheng Shao }
\address{Department of Mathematics\\ University of Kentucky\\ 757 Patterson Office Tower\\ Lexington, KY, 40506 \\ USA}
\email{Xuancheng.Shao@uky.edu}
\thanks{X.S.~was supported by a Glasstone Research Fellowship.}
 
\thanks{We would like to thank Dimitris Koukoulopoulos for a valuable discussion, as well as the referees for their careful readthroughs and comments}


\date{\today}

\begin{abstract}  
Part-and-parcel of the study of ``multiplicative number theory'' is the study of the distribution of multiplicative functions in arithmetic progressions. Although appropriate analogies to the  Bombieri-Vingradov Theorem have been proved for  particular examples of multiplicative functions, there has not previously been headway on a  general theory; seemingly none of the different proofs of the  Bombieri-Vingradov Theorem for primes  adapt well to this situation. In this article we find out why such a result has been so elusive, and discover what can be proved along these lines and develop some limitations. For a fixed residue class $a$ we extend such averages out to moduli $\leq x^{\frac {20}{39}-\delta}$.
\end{abstract}

\maketitle

\section{Formulating the Bombieri-Vingradov Theorem for multiplicative functions}

The Bombieri-Vingradov Theorem, a mainstay of analytic number theory, shows that the prime numbers up to $x$  are ``well-distributed'' in all arithmetic progressions mod $q$, for almost all integers $q\leq x^{1/2-\varepsilon}$. To be more precise, for any given sequence $f(1), f(2), \ldots$, we define
 $$
\Delta(f,x;q,a):=\sum_{\substack{n\leq x \\ n\equiv a \pmod q}} f(n)  -  \frac 1{\varphi(q)} \sum_{\substack{n\leq x \\ (n,q)=1}} f(n).
$$
The Bombieri-Vinogradov Theorem states that if $f$ is the characteristic function for the primes, then  for any given $A>0$ there exists $B=B(A)>0$ such that if $Q\leq x^{1/2}/(\log x)^B$ then 
\begin{equation}\label{eq:BVI}
\sum_{q\sim Q} \max_{a:\ (a,q)=1} | \Delta(f,x;q,a) | \ll \frac x {(\log x)^A},
\end{equation}
where, here and henceforth, ``$q\sim Q$'' denotes the set of integers $q$ in the range $Q<q\leq 2Q$. The analogous result is known to hold when $f=\mu$, the Mobius function, and when $f$ is the characteristic function for the $y$-smooth numbers~\cite{FT, Har}, and the literature is swarming with many other interesting examples besides. There are many proofs of the original Bombieri-Vinogradov Theorem: more modern proofs rely on bilinearity and Vaughan's identity, for example see chapter 28 of \cite{Da}, or the very elegant proof  in Theorems 9.16, 9.17, and 9.18  of  \cite{FIOpera}. The extraordinary generality of the latter proof leads one to guess that something like the Bombieri-Vinogradov Theorem should be true for most sensible arithmetic functions $f$. 

In this paper we focus on proving a Bombieri-Vinogradov type Theorem  for   multiplicative functions $f$ which take values within the unit circle, something that has been proved for several interesting examples, and one might guess is true in some generality. However one needs to be careful: if $f(n)=(n/3)$, the quadratic character mod 3, then $f$ is certainly not well-distributed in arithmetic progressions mod 3, nor in arithmetic progressions mod  $q$, whenever 3 divides $q$. More generally if $f=\chi$ is a primitive character mod $r$, or even if $f$ is ``close'' to $\chi$, then $f$ is  not well-distributed in arithmetic progressions  mod  $q$, whenever $r$ divides $q$. So this is a significant departure from the classical Bombieri-Vinogradov type Theorem, in that it seems likely that\footnote{This is  discussed in detail in section \ref{ClaimJusitify}, and   proved there up to a factor of $\log x$.}
\[
\sum_{q\sim Q} \max_{a:\ (a,q)=1} | \Delta(f,x;q,a) | \gg \max_{\substack{\chi \pmod r \\ \chi \ \text{primitive} \\ r>1}} \frac 1{\varphi(r)} |S_f(x,\chi)|,
\]
where
\[
 S_f(x,\chi):= \sum_{n\leq x} f(n) \overline{\chi}(n) .
\]
Therefore if $f$ strongly correlates with some $\chi$ of ``small" conductor 
(that is, $S_f(x,\chi)\gg x/(\log x)^C$ for some non-principal character $\chi$ of conductor $\ll (\log x)^D$, for  for some fixed $C,D>0$) then  \eqref{eq:BVI} cannot hold for all $A$. So to prove something like \eqref{eq:BVI} we need to  assume that no such character exists. The usual way to formulate this involves equidistribution for moduli of small conductor (see e.g. (9.68) in \cite{FIOpera}): For any given $A>0$ we have
\medskip

\noindent \textbf{The $A$-Siegel-Walfisz criterion}: We say that $f$ \emph{satisfies the $A$-Siegel-Walfisz criterion} if
\[
\sum_{\substack{n\leq x \\ n\equiv a \pmod q \ }} f(n)  -  \frac 1{\varphi(q)} \sum_{\substack{n\leq x \\ (n,q)=1}} f(n) \ll_A \frac x {(\log x)^A}
\]
for all $(a,q)=1$ and $x \geq 2$.
We say that $f$ \emph{satisfies the Siegel-Walfisz criterion} if this holds for any fixed $A>0$.

Let
\[
F(s) = \sum_{n=1}^{\infty} \frac{f(n)}{n^{s}}\ \ \text{and} \ \
- \frac{F'(s)}{F(s)} = \sum_{n=2}^{\infty}   \frac{\Lambda_f(n)}{n^{s}},  
\]
for Re$(s)>1$.  Following \cite{GHS}, we restrict  attention to the class ${\mathcal C} $ of multiplicative functions $f$ for which
\[
 |\Lambda_f(n)|\leq   \Lambda(n) \ \ \textrm{for all} \ \ n\geq 1.
 \]
This includes  most multiplicative functions of interest, including all $1$-bounded completely multiplicative functions. Two key observations are that if $f\in \mathcal C$ then each $|f(n)|\leq 1$, and if $F(s)G(s)=1$ then $g\in \mathcal C$.
 
 \begin{corollary} \label{Thm: SW}  Let $f$ be a  multiplicative function with $f\in \mathcal C$, and assume that $f$ satisfies the 
 $1$-Siegel-Walfisz criterion. Fix $\delta, \varepsilon>0$.
 If  $Q \leq x^{1/2-\delta}$ then  
\[ 
\sum_{q\sim Q} \max_{a:\ (a,q)=1} | \Delta(f,x;q,a) | \ll    \frac{x} {(\log x)^{1-\varepsilon}}  .
 \]
\end{corollary}

This bound is considerably weaker than the hoped-for bound \eqref{eq:BVI} in that we improve upon the ``trivial bound'', $\ll x$, by only a factor $(\log x)^{1-\varepsilon}$ rather than by an arbitrary power of $\log x$. However we will show in Proposition~\ref{LargePrimes2b}  that Corollary \ref{Thm: SW} is, up to the $\varepsilon$-factor, best possible.

 \subsection{Taking exceptional characters into account}
Even if the $1$-Siegel-Walfisz criterion does not hold we can  prove a version of the Bombieri-Vinogradov Theorem which takes account of the primitive characters $\psi$ for which $S_f(x,\psi)$ is large.

For any character $\psi \pmod r$, define
\[
\sigma_f(x,\psi) := \sup_{x^{1/2}< X \leq x} |S_f(X,\psi)/X|.
\]
We order the  primitive characters $\psi_1 \pmod {r_1},\ \psi_2 \pmod {r_2},\ldots $ with each $r_i\leq \log x$, so that 
\[
\sigma_f(x,\psi_1)  \geq \sigma_f(x,\psi_2)  \geq \ldots
\]
which more-or-less corresponds to the ordering of $|S_f(x,\psi)|$, at least if these get ``large''.

Notice that $ \sum_{ {n\leq x,\ (n,q)=1}} f(n) =  S_f(x,\chi_0)$, and so
\begin{equation} \label{expansion}
\Delta(f,x;q,a)=\   \frac 1{\varphi(q)} \sum_{\substack{ \chi \pmod q \\ \chi \ne \chi_0 }}  \chi(a) S_f(x,\chi).
\end{equation}
It therefore makes sense to reformulate the Bombieri-Vinogradov Theorem, so as to remove the largest $k$ value(s) of $|S_f(x,\chi)|$ from the sum in \eqref{expansion}.\footnote{From this perspective, the usual formulation  implicitly assumes that $\psi_1=1$, which is true whenever $f(n)\geq 0$ for all $n$.}  If $|S_f(x,\chi)|$ is large then $\chi \pmod q$ is induced from some $\psi_j$ with $1\leq j\leq k$; and there is a (unique) character $\chi_j \pmod q$  induced by   $\psi_j$, if and only if $r_j|q$. Therefore we  define
 $$
\Delta_k(f,x;q,a):=\sum_{\substack{n\leq x \\ n\equiv a \pmod q}} f(n)  -  \frac 1{\varphi(q)} \sum_{\substack{ 1\leq j\leq k \\ r_j|q}} \chi_j(a) S_f(x,\chi_j).
$$
For fixed $k$, we believe that 
\[
\sum_{q\sim Q} \max_{a:\ (a,q)=1} | \Delta_k(f,x;q,a) | \gg \max_{j\geq k+1}\ \frac 1{\varphi(r_{j})} |S_f(x,\psi_{j})| 
\]
(see the discussion in section \ref{ClaimJusitify}).
Moreover $f$ can be chosen so that this is $\gg x/\log x$ for any fixed $k$ (see section \ref{ManyChar}), unconditionally. Nonetheless this reformulation, taking into account the characters that correlate well with $f$, can lead to upper bounds which are almost of this strength, as we now state.

 \begin{theorem} \label{Cor:Result2} Fix $\delta, \varepsilon>0$ and let $k$ be the largest integer $\leq 1/\varepsilon^2$.
For any   $f\in {\mathcal C} $, and for any  $Q \leq x^{1/2-\delta}$, we have  
\[ 
\sum_{q\sim Q} \max_{a:\ (a,q)=1} | \Delta_k(f,x;q,a) | \ll    \frac{x} {(\log x)^{1-\varepsilon}} .
 \]
\end{theorem}

Corollary \ref{Thm: SW} will follow from Theorem \ref{Cor:Result2}.

 Theorem \ref{Cor:Result2} is close to ``best possible"  in that (as we show in section~\ref{ManyChar}),   for   given integer $k$, there is an  
$ \varepsilon' \asymp \frac 1{\sqrt{k}}$, for which there exists $f\in {\mathcal C} $ such that 
\[ 
\sum_{q\sim Q} \max_{a:\ (a,q)=1} | \Delta_{k}(f,x;q,a) | \gg   \frac{x} {(\log x)^{1-\varepsilon'}} .
\]

In \cite{GHS} it is proved that if $\log q\leq (\log x)^\delta$ where $\delta=\delta(\varepsilon)>0$, then
\begin{equation} \label{eq:GHS}
  | \Delta_k(f,x;q,a) | \ll  \frac 1{\varphi(q)}  \frac{x} {(\log x)^{1-\varepsilon}} 
\end{equation}
whenever $(a,q)=1$, where $k$ is the largest integer $\leq 1/\varepsilon^2$.\footnote{The definition of $\Delta_k(f,x;q,a)$ in \cite{GHS} involves removing the $k$ largest character sums for characters mod $q$, whereas here we use the characters mod $q$ induced from the largest character sums for characters mod $r$, for any $r|q$ with $r\leq \log x$. This is a minor technical difference, which is sorted out in Corollary~\ref {No Exceptions}.}  That is,  $f$ is well-distributed  in all arithmetic progressions mod $q$ for \emph{all} $q\leq Q$ provided $x$ is very large compared to $Q$, and in Theorem \ref{Cor:Result2} we have   shown good distribution in all arithmetic progressions mod $q$ for \emph{almost all} $q\leq Q$ provided $x>Q^{2+\delta}$, a much larger range for $q$  but at the cost of some possible exceptions. This is reminiscent of what we know about the distribution of prime numbers in arithmetic progressions, though here we have significantly weaker bounds on the  error terms.

 \subsection{Limitations on the possible upper bounds}
One might guess that if there are no characters of small conductor that strongly correlate with $f$ then perhaps one can significantly improve the upper bounds in Theorem  \ref{Cor:Result2} and Corollary  \ref{Thm: SW} (if one assumes an $A$-Siegel-Walfisz criterion for all $A>0$). Unfortunately there is another obstruction, in which the values of $f(p)$ for the large primes $p$ up to $x$, do their utmost to block equi-distribution:

 \begin{proposition} \label{LargePrimes2} Let $g$ be a  multiplicative function  with each   $|g(n)|\leq 1$, and suppose we are given $Q\leq x$ with $Q,x/Q\to\infty$ as $x\to \infty$. There exists a subset  $\mathcal P$ of the primes  in the interval $(x/2,x]$, that contains almost all of those primes,\footnote{Here ``almost all of the primes" in $(x/2,x]$ means $(1+o(1))(\pi(x)-\pi(x/2))$ primes.}
 and a constant $\sigma\in \{ -1,1\}$, such that if  $f(n)=g(n)$ for all $n\leq x$ other than   $n\in \mathcal P$, and  $f(p)=\sigma$ for all $p\in \mathcal P$, then $| \Delta(f,x;q,1) |\gg \pi(x)/\varphi(q)$ for  at least half of the moduli  $q\sim Q$. This implies that
\begin{equation}\label{LowerBd}
\sum_{q\sim Q}  | \Delta(f,x;q,1) | \gg  \frac{x} {\log x } . 
\end{equation}
\end{proposition}
 
It is widely believed that for any fixed $\ee, A>0$ we have
\begin{equation} \label{ExtendPNT}
\pi(x;q,a) = \frac{\pi(x)}{\varphi(q)} \left( 1 +O\left(  \frac 1{(\log x)^A} \right) \right)
\end{equation}
 whenever $(a,q)=1$ and $q\leq x^{1 -\varepsilon}$.

We now show that even assuming a strong Siegel-Walfisz criterion, we expect that one cannot significantly improve 
the upper bound in Corollary \ref{Thm: SW}.

\begin{proposition} \label{LargePrimes2b} Assume \eqref{ExtendPNT}. Let $x^{2/5} < Q < x/2$. There exists a completely multiplicative function $f$, taking only values $-1$ and $1$, which satisfies the $A$-Siegel-Walfisz criterion for all $A>0$, but for which \eqref{LowerBd} holds.
 \end{proposition}

We have exhibited two fundamental obstructions to a Bombieri-Vinogradov Theorem for multiplicative functions:

 (i)\ If $f$ correlates closely with a character of small conductor; or
 
 (ii)\ If the values of $f(p)$ with $x/2<p\leq x$ conspire against equi-distribution.
 
 \noindent Consequently, although we have been able to beat the ``trivial bound'' by a factor of $(\log x)^{1-\varepsilon}$ by taking (i) into account in Theorem \ref{Cor:Result2}, (ii) ensures that we cannot do much better in general.
 
  These two obstructions have arisen  before in the multiplicative functions literature, in Montgomery and Vaughan's seminal work \cite{MV}  on bounding exponential sums twisted by multiplicative coefficients. For this question, the contributions from obstruction (i) are  identified precisely in  \cite{BGS}, but the sharpness of the bounds are inevitably restricted by obstruction (ii). However, in Proposition 1 of \cite{dlB}, de la Bret\`eche showed that one can obtain much better bounds if one restricts attention to $f$ that are supported only on smooth numbers,\footnote{\emph{Smooth numbers} are  integers with no large prime factors.  That is, we restrict attention to multiplicative function $f(.)$ for which $f(p^k)=0$ for any prime $p>y$ and integer $k\geq 1$.}   since then obstruction (ii) is rendered irrelevant. We can do much the same here:

 \subsection{Multiplicative functions supported on smooth numbers} The following key result generalizes
 \eqref{eq:GHS} (which was proved in~\cite{GHS}) to error terms in which one ``saves'' an arbitrary power of $\log x$, for multiplicative functions supported on smooth numbers.
 
 Given a finite set of primitive characters, $\Xi$,   let $\Xi_q$ be the set of characters mod $q$ that are induced by the characters in $\Xi$, and then define  
 \[
\Delta_\Xi(f,x;q,a):=  \sum_{\substack{n\leq x \\ n\equiv a \pmod q}} f(n) -   \frac 1{\varphi(q)} \sum_{ \substack{  \chi  \in   \Xi_q }} \chi(a) S_f(x,\chi) .
\]

\begin{proposition} \label{MainCor} Fix $\varepsilon>0$  and $0<\gamma\leq \frac 12-\varepsilon$, let  
$y=  x^{\gamma}$, and 
suppose that $f\in \mathcal C$, and is only supported on $y$-smooth numbers. Fix $B\geq 0$. 
There exists a set, $\Xi$, of primitive characters $\psi \pmod r$ with each $r\leq R :=  x^{  \varepsilon/(3\log\log x)}$, containing $\ll   (\log x)^{6B+7+ o(1)}$ elements, such that  if $q\leq R$ and $(a,q)=1$ then
\[    
| \Delta_{\Xi}(f,x;q,a) | \ll  \frac 1{\varphi(q)}  \frac{x} {(\log x)^{B}} .
\]
\end{proposition}

By  \eqref{eq:GHS} we know that if $B<1$ then Proposition \ref{MainCor} holds for any $f\in \mathcal C$ (not just those 
supported on $y$-smooth numbers) with the size of $|\Xi|$ bounded only in terms of $B$.

We state two Bombieri-Vinogradov type Theorems that follow from this.

 \begin{theorem} \label{Keep Xi} 
Fix $0< \varepsilon\leq \frac 1{10}, A, B$ with $B>0$ and $2A>6B+7$. Let  $y=x^\varepsilon$. 
Let $f\in \mathcal C$ be a   multiplicative function which is only supported on $y$-smooth integers. There exists a set, $\Xi$, of primitive characters, containing $\ll   (\log x)^{6B+7+ o(1)}$ elements, such that for any $Q\leq x^{1/2}/y^{1/2}(\log x)^A$, we have
 \begin{equation}\label{eq:boum5}
 \sum_{q \sim Q} \max_{(a,q)=1} \left|  \Delta_{\Xi}(f,x;q,a)  \right|  \ll_B \frac x{(\log x)^B}.
 \end{equation}
\end{theorem}

\begin{corollary} \label{MathResult3}
Fix $0 < \varepsilon \leq \frac{1}{10}$. Let $y = x^{\varepsilon}$.
Let $f\in \mathcal C$ be a   multiplicative function which is only supported on $y$-smooth integers.
Assume that the Siegel-Walfisz criterion holds for $f$.
For any given $B>0$ there exists $A$ such that  for any $Q\leq x^{1/2}/(y^{1/2}(\log x)^A)$, we have
\[ \sum_{q \sim Q} \max_{(a,q)=1} \left|  \Delta(f,x;q,a)  \right|    \ll_B \frac x{(\log x)^B}. \]
\end{corollary}
 
Note that these bounds are non-trivial since the number of $x^{\varepsilon}$-smooth integers up to $x$ is $\gg x$.  Combining Corollary~\ref{MathResult3} with the machinery developed in~\cite{Sha}, one may prove for such multiplicative functions that their higher Gowers $U^k$-norms are $o(1)$ in progressions on average. This result will be stated and discussed in Section~\ref{sec:higher-uk}.

 \subsection{Breaking the $x^{1/2}$-barrier}

The main method used in our proofs is a  modification of that developed by Green in \cite{Gre}; see also~\cite{Sha} for using a similar argument to deal with higher Gowers norms. Green proved (a more general result which implies) that 
  \[ 
\sum_{\substack{q\sim Q \\ q \ \text{prime}}}  | \Delta(f,x;q,1) | \ll  \frac{x \log\log x} {(\log x)^2 } ,
 \]
 for any $Q<x^{\frac{20}{39}-\varepsilon}$, remarkably breaking the $x^{1/2}$-barrier. In comparison, previous results concerning breaking the $x^{1/2}$-barrier in the original Bombieri-Vinogradov theorem typically only work when $q$ is required to be ``smooth"~\cite{BFI1, FI, Zh}, or else only beat the $x^{1/2}$-barrier by $x^{o(1)}$~\cite{BFI2,BFI3}. See also~\cite{Dr} and the references therein for results along the same line when $f$ is the indicator function of smooth numbers.

In Green's result, the issue of correlations with non-trivial characters of small conductor  does not arise since no such character induces a character modulo a large prime (and Green is only summing over prime moduli). Nonetheless obstruction (ii) still applies and so Proposition \ref{LargePrimes2}, as well as the construction in section \ref{NoBV}, shows that Green's result is more-or-less best possible (up to the $\log\log x$ factor).
One can modify Green's proof to include composite moduli by taking account of the characters $\psi_j$, as we have done here.
This leads to the following extensions (for fixed $a$) of Corollary \ref{Thm: SW} and Theorems \ref{Cor:Result2}, as well as Theorem~\ref{Keep Xi} and Corollary~\ref{MathResult3}.

 \begin{theorem} \label{Cor:Result2+}  Let $f$ be a  multiplicative function with $f\in \mathcal C$.
Fix $\delta, \varepsilon>0$ and let $k$ be the largest integer $\leq 1/\varepsilon^2$.
For any     $1\leq |a|\ll  Q\leq x^{\frac{20}{39}-\delta}$, we have  
\[ 
\sum_{\substack{q\sim Q \\ (a,q)=1}}   | \Delta_k(f,x;q,a) | \ll    \frac{x} {(\log x)^{1-\varepsilon}} .
 \]
If $f$ satisfies the  $1$-Siegel-Walfisz criterion then  
\[ 
\sum_{\substack{q\sim Q \\ (a,q) = 1}}   | \Delta(f,x;q,a) | \ll    \frac{x} {(\log x)^{1-\varepsilon}}  .
 \]
\end{theorem}

\begin{theorem} \label{MathResult3+} 
 Fix $\delta, B>0$. Let $y = x^{\varepsilon}$ for some $\varepsilon > 0$ sufficiently small in terms of $\delta$. Let $f\in \mathcal C$ be a   multiplicative function which is only supported on $y$-smooth integers. Then there exists a set, $\Xi$, of primitive characters, containing $\ll   (\log x)^{6B+7+ o(1)}$ elements, such that for any $1 \leq |a| \ll Q \leq x^{\frac{20}{39}-\delta} $, we have
 \begin{equation*}
 \sum_{\substack{q \sim Q \\ (a,q) = 1}} \left|  \Delta_{\Xi}(f,x;q,a)  \right|  \ll \frac x{(\log x)^B}.
 \end{equation*}
If $f$ satisfies the   Siegel-Walfisz criterion then   
 \begin{equation} \label{eq: FT2}
  \sum_{\substack{q \sim Q \\ (a,q)=1}}   \left|  \Delta(f,x;q,a)  \right|    \ll \frac x{(\log x)^B}.
   \end{equation}
\end{theorem}

The proofs of these last two results, which   break the $x^{1/2}$-barrier, rely on a deep estimate of Bettin and Chandee \cite{BC} on bilinear Kloosterman sums, which is an impressive development going beyond the famous estimates of Duke, Friedlander and Iwaniec \cite{DFI}. It should be noted that Fouvry and Tenenbaum (Th\'eor\`eme 2 in \cite{FT2}) established
\eqref{eq: FT2} unconditionally when $f$ is the characteristic function of the $y$-smooth integers for any $Q\leq x^{3/5-\delta}$, and   any $y\leq x^\epsilon$.

Our focus in this last part of the paper is to go beyond the $x^{1/2}$-barrier by incorporating the necessary expedient of $x^\varepsilon$-smooth functions into our arguments. Can one go much further beyond the $x^{1/2}$-barrier using current technology, especially if $f$ is $y$-smooth for $y$ a lot smaller than $x^\varepsilon$?
In a sequel to this paper,  joint with Sary Drappeau, we will extend Theorem \ref{MathResult3+}  to the range $Q\leq x^{3/5-\varepsilon}$ and with a wide range for the smoothness parameter $y$, by incorporating somewhat different techniques into our arguments, and improving the $x$ in the upper bound to $\Psi(x,y)$, the counting function for the $y$-smooths.

Probably the most novel part of our work is to compare the mean value of $f$ in the arithmetic progression $a \pmod q$ with the the mean value of $f$ in the arithmetic progression $a \pmod {q_s}$, where $q_s$ is the largest $w$-smooth divisor of $q$. See, for example,  Theorem \ref{MathResult2}.

There is a series of seven papers by Elliott on multiplicative functions in arithmetic progressions, some of which explore Bombieri-Vinogradov type theorems (particularly \cite{Ell4, Ell5}). Several of the themes in this paper have their origins in his seminal work.

\subsection*{Notation}
Let $w = w(x)$ be a parameter, which will typically be a fixed power of $\log x$. 
For any positive integer $q$, we have a unique decomposition $q = q_sq_r$ of $q$ into a $w$-smooth part $q_s$ and a $w$-rough part $q_r$, where $q_s = (q, \prod_{p\leq w}p^{\infty})$ is the largest $w$-smooth integer dividing $q$, and $q_r = q/q_s$ has no prime factors $\leq w$. Although the values of $q_r,q_s$ depend on the parameter $w$, we will not explicitly indicate this dependence as the choice of $w$ should always be clear from the context.

\section{Smooth number estimates}\label{sec:smooth-prelim}

We call $n$ a $y$-\emph{smooth integer} if all of its prime factors are $\leq y$. We let $P(n)$ denote the largest prime factor of $n$ so that $n$  is $y$-smooth if and only if $P(n)\leq y$.

We need several well-known estimates involving the distribution of smooth numbers (unless otherwise referenced, see \cite{Gra}). Let $\Psi(x,y)$ be the number of $y$-smooth integers up to $x$.   If $y\leq (\log x)^{1+o(1)}$ then $\Psi(x,y)=x^{o(1)}$. Otherwise if $x\geq y\geq  (\log x)^{1+\varepsilon}$
we write $x=y^u$ and then 
\[
\Psi(x,y) = x /u^{u+o(u)} .
\]
In particular if $y=(\log x)^A$ then $\Psi(x,y) =x^{1-\frac 1A + o(1)}$. Key consequences include if $x\geq y$ then 
\[
\sum_{\substack{ x<n\leq 2x \\ P(n)\leq y }} \frac 1n,\quad \frac 1{\log y} \sum_{\substack{ n>x \\ P(n)\leq y }} \frac 1n = u^{-u+o(u)}  + x^{-1+o(1)}.
\]
One consequence of this is that if $Y\geq w^{(C+\varepsilon) \log\log x / \log\log\log x}$ then
\begin{equation} \label{tail sum}
\sum_{\substack{  n>Y \\ P(n)\leq w }} \frac xn \ll \frac x{(\log x)^C}.
\end{equation}

Rather more precisely, we define $\rho(u)=1$ for $0\leq u\leq 1$, and then determine $\rho(u)$ from the differential-delay equation $\rho'(u)=-\rho(u-1)/u$ for all $u>1$. Then 
\begin{equation} \label{HildEst}
\Psi(x,y) = x\rho(u) \left( 1+O\left( \frac {\log (u+1)}{\log y}\right)\right)
\end{equation}
for $x\geq y\geq \exp( (\log\log x)^2)$.

Define $\alpha(x,y)$ to be the real number for which 
\[
\sum_{p\leq y} \frac{\log p}{p^\alpha-1} = \log x;
\]
one has $y^{1-\alpha}\asymp u\log u$ when $y \gg \log x$. If $x\geq y\geq  (\log x)^{1+\varepsilon}$ then $\alpha\gg \varepsilon$.
We need the comparison bounds
\begin{equation} \label{alphaCompare1}
\Psi(x/d,y)= \left( 1 + O \left( \frac 1u \right) \right)  \frac{\Psi(x,y)} {d^{\alpha}}
\end{equation}
for $d=y^{O(1)}$ and, in general,
\begin{equation} \label{alphaCompare}
\Psi(x/d,y) \ll \frac{\Psi(x,y)} d \cdot d^{1-\alpha}
\end{equation}
which follow from Theorem 2.4 of \cite{BT}.
This last bound implies  that if $y \geq (\log x)^{1+\varepsilon}$ then $\int_2^x \Psi(t,y)/t \ dt \ll_{\varepsilon} \Psi(x,y)$.
The bound in \eqref{alphaCompare1} is not useful for us when $u,d\ll 1$. In this range we have
\begin{equation} \label{alphaCompare2}
\Psi(x/d,y) =\frac { \Psi(x,y) }{ d^{\alpha+O(1/\log x)}}  \left( 1 +O\left(\frac{\log u}{\log y} \right) \right)
\end{equation}
by (2.22) and (2.23) of~\cite{BT}.


  Theorem 6    of \cite{FT} gives a version of the Bombieri-Vinogradov Theorem for $y$-smooth numbers (though see also \cite{Wol}, and see~\cite{Har} for an improvement on the range of $y$):\ For any $A>0$ there exists a constant $B=B(A)$ such that 
\begin{equation} \label{BVsmooth}
\sum_{q\leq \sqrt{x}/(\log x)^B} \max_{(a,q)=1} \left|  \Psi(x,y;q,a) - \frac {\Psi_q(x,y)}{\varphi(q)} \right| \ll \frac x{(\log x)^A}
\end{equation}
where $ \Psi_q(x,y)$ denotes the number of $y$-smooth integers up to $x$ that are coprime to $q$, and $ \Psi(x,y;q,a)$ denotes those that are $\equiv a \pmod q$. Then (4.11) of~\cite{BT} gives the upper bound
\begin{equation} \label{SmoothUB}
 \Psi_q(x,y) \ll  \frac  {\varphi(q)} q \Psi (x,y)  
\end{equation}
provided $x\geq y\geq \exp((\log\log x)^2)$ and $q\leq x$.

 Corollary 2 of \cite{Hil} implies a good  upper bound from smooth numbers in short intervals: For any fixed $\kappa>0$,
\begin{equation} \label{SmoothShorts}
\Psi(x+\frac xT,y)-\Psi(x,y) \ll_\kappa \frac {\Psi(x,y)}T \text{  for  } 1\leq T\leq \min \{ y^\kappa, x\}.
\end{equation}

\section{The contribution of characters}   \label{sec:contrib-char}

 \subsection{Comparing large character sums for  a primitive character and the characters it induces}
Recall the definition of $\sigma_f(x,\psi)$ from the introduction. We now define
\[
\sigma_f(x,z,\psi) := \sup_{x/z< X \leq x} |S_f(X,\psi)/X|,
\]
so that $\sigma_f(x,\psi) =\sigma_f(x,x^{1/2},\psi) $.

\begin{lemma} \label{Lem:Chars} Suppose that $f \in \mathcal{C}$.  Let $z\geq \exp((\log\log x)^2)$ and $Q \leq x$.
If $\chi \pmod q$ is induced by $\psi\pmod r$, where $r \leq q \leq Q$, then 
\[
S_f(x,\chi) \ll _C  x \sigma_f(x,z,\psi)  \log\log x +\frac x{(\log x)^C}  ,
\]
and
\[
S_f (x,\psi) \ll _C  x  \sigma_f(x,z,\chi)   \log\log x +\frac x{(\log x)^C} ,
\]
for any given constant $C>0$. Furthermore, for $\psi\pmod{r}$ with $r \leq Q$ we have
\[
 \sum_{\substack{r|q\sim Q  \\ \chi \mod   q \ \text{induced by} \ \psi}}  \frac{|S_f(x,\chi)|}{\varphi(q)} \ll_C
 \left(\sigma_f(x,z,\psi)\    +\frac 1{(\log x)^C}\right)  \frac{x}{\varphi(r)}  ;
\]
and if $P(r)\leq w$ then
\[
\frac 1{\log w}\sum_{\substack{r|q\leq Q ,\ P(q)\leq w  \\ \chi \mod   q \ \text{induced by} \ \psi}}  \frac{|S_f(x,\chi)|}{\varphi(q)} \ll_C
 \left(\sigma_f(x,z,\psi)\    +\frac 1{(\log x)^C}\right)  \frac{x}{\varphi(r)}    .
\]
\end{lemma}

\begin{proof}   Let $h(.)$ be the multiplicative function which is supported only on the prime powers $p^k$, for which $p$ divides $q$ but not $r$, with $(h*f\overline\psi)(p^k)=0$ for these $p^k$. Thus $h*f\overline\psi = f\overline\chi$, and note that $h\in \mathcal C$ as $f\in \mathcal C$, so that each $|h(m)|\leq 1$. Now
\[
S_f(x,\chi) = \sum_{m\leq x} h(m) S_f(x/m,\psi)
\]
and therefore we obtain, as $|S_f(x/m,\psi)|\leq \sigma_f(x,z,\psi) x/m$ if $m\leq z$, 
\[
\begin{split}
|S_f(x,\chi)| & \ll \sigma_f(x,z,\psi)\ x \sum_{\substack{m\leq z \\ p|m\implies p|q, \ p\nmid r } } \frac{1}m + x \sum_{\substack{z<m\leq x\\ p|m\implies p|q } }  \frac{1}m\\
& \ll_C \sigma_f(x,z,\psi)\ x \cdot \prod_{p|q,\ p\nmid r} \frac p{p-1} + \frac{x}{z^{\alpha}} \sum_{p|m\implies p|q} \frac{1}{m^{1-\alpha}} ,
  \end{split}
\]  
where we applied Rankin's trick to the second sum with $\alpha = 2C/\log\log x$. Both terms in the bound are maximized when $q$ is the product of the primes $\ll \log Q \leq \log x$, in which case  the Euler product is $\ll \log\log Q$ and
\[
\begin{split} 
\frac{1}{z^{\alpha}} \sum_{p|m\implies p|q} \frac{1}{m^{1-\alpha}} &\leq \frac{1}{z^{\alpha}} \prod_{p \leq \log x  } \left(1 - \frac{1}{p^{1-\alpha}}\right)^{-1} \leq \frac{1}{z^{\alpha}} \prod_{p \leq \log x} \left(1 - \frac{e^{2C}}{p}\right)^{-1} \\
&\ll \frac{(\log\log x)^{O_C(1)}}{z^{\alpha}} \ll \frac{x}{(\log x)^C}.
\end{split}
\]

 In the other direction we have
\[
S_f(x,\psi) = \sum_{\substack{m\leq x \\ p|m\implies p|q, \ p\nmid r } } f(m) \overline{\psi(m)} S_f(x/m,\chi)
\]
and the same argument leads to the second claimed result.

We now prove the fourth part of the Lemma (the third part is proved by a simple modification of this proof), by using the upper bound proved for $|S_f(x,\chi)|$ in the first part. The second term in the upper bound is, writing $q=rn$,
\[
\ll  \frac{1}{\log w} \sum_{ \substack{ r|q\leq Q\\   P(q)\leq w}}  \frac{1}{\varphi(q)}  \frac x{(\log x)^C} 
\ll  \frac{1}{\log w} \sum_{ \substack{ n\leq Q/r\\   P(n)\leq w}}  \frac{1}{\varphi(n)} \cdot  \frac x{\varphi(r)(\log x)^C} \ll \frac x{\varphi(r)(\log x)^C}.
\] 
The first term in the upper bound is
\[
\begin{split}
&  \ll  \frac{1}{\log w} \sum_{ \substack{ r|q\leq Q\\   P(q)\leq w}}  \frac{1}{\varphi(q)} \left( \sigma_f(x,z,\psi)\ x\sum_{\substack{m\leq z \\ p|m\implies p|q, \ p\nmid r } } \frac{1}m \right)\\
 & \leq    \sigma_f(x,z,\psi)\   \frac{1}{\log w}   \sum_{ \substack{ r|q\leq Q\\   P(q)\leq w}}  \frac{x}{\varphi(q)}\sum_{\substack{m\leq z \\ m|q, \ (m,r)=1} } \frac{\mu^2(m)}{\varphi(m)}    \\
 & \leq   \sigma_f(x,z,\psi)\ x \sum_{\substack{m\leq z \\  (m,r)=1 } }  \frac{\mu^2(m)}{\varphi(m)}  
\frac{1}{\log w}   \sum_{ \substack{ mr|q\leq Q\\  P(q)\leq w}}   \frac{1}{\varphi(q)}     \\
  & \ll   \sigma_f(x,z,\psi)\ \frac{x}{\varphi(r)} \sum_{\substack{m\leq z \\  (m,r)=1 } }  \frac{\mu^2(m)}{\varphi(m)^2}   \frac{1}{\log w}   \sum_{ \substack{ n\leq Q/mr\\ P(n)\leq w}}   \frac{1}{\varphi(n)}   \ll   \sigma_f(x,z,\psi)\ \frac{x}{\varphi(r)}   ,  \\
  \end{split}
  \]
writing $q=mrn$,  and the claim follows.
\end{proof}

 \subsection{Focusing on large character sums}\label{sec:large-char-sum-2}
For fixed $B>0$, let $\Xi(B,Q)$ denote the set of primitive characters $\psi \pmod r$ with $r\leq Q$ for which 
 \[
 \sigma_f(x,\psi) \geq \frac 1{(\log x)^B}.
 \]
  
\begin{corollary} \label{No Exceptions} Let $f\in \mathcal C$ and $B > 0$. (a)\ Suppose that $Q \leq x$. If $\chi \pmod q$ is a character with $q\leq Q$ and  is not induced by any of the characters in $\Xi(B,Q)$, then 
\[
 S_f(x,\chi) \ll \frac {x\log\log x}{(\log x)^B}.
 \]
(b)\ Now suppose that $\log Q=(\log x)^{o(1)}$ and $J\geq 2$ is a given integer with $B<1-1/\sqrt{J}$. Then 
$|\Xi(B,Q)|<J$ and 
\begin{equation}\label{GHSprecise}
|\Delta_{\Xi(B,Q)}(f,x;q,a)| \ll \frac 1{\varphi(q)} \frac x {(\log x)^{B+o(1)}}
\end{equation}
for any $q \leq Q$ and $(a,q) = 1$.
\end{corollary}

\begin{proof} (a)\  If $\chi$ is induced from $\psi$ then $ \sigma_f(x,\psi) \leq 1/{(\log x)^B}$ by the hypothesis, 
and the result then follows from the first part of  Lemma \ref{Lem:Chars}.

(b)\ Suppose that there are at least $J$ characters $\psi_j \pmod {r_j}$ in $\Xi(B,Q)$. Let $r=[r_1,\ldots,r_J]$ so that $\log r = (\log x)^{o(1)}$, and let $\chi_j$ be the character mod $r$ induced by $\psi_j$, so that, for each $j$, there exists $x^{1/2}<X_j\leq x$ for which
$\sigma_f(X_j,\chi_j)\gg 1/ (\log x)^{B+o(1)}$ by  the second part of  Lemma \ref{Lem:Chars}. However, by Theorem 6.1 of \cite{GHS}, one of these is $\ll 1/(\log x)^{1-1/\sqrt{J}+o(1)}$, a contradiction.

Now Theorem 6.1 of \cite{GHS}, applied to the set $S$ of $J-1$ characters $\chi \pmod q$ which give the $J-1$ largest values of   $|S_f(x,\chi)|$, implies that
\[
|\Delta_{S}(f,x;q,a)| \ll \frac 1{\varphi(q)} \frac x {(\log x)^{1-1/\sqrt{J}+o(1)}} .
\]
Write $\Xi = \Xi(B, Q)$. Now $ |S_f(x,\chi)| \ll x/{(\log x)^{B+o(1)}}$ for every $\chi\in S \setminus \Xi_q$ by (a), and also for every $\chi\in \Xi_q\setminus S$ by the  definition of $S$, Theorem 6.1 of \cite{GHS}, and the hypothesis $B < 1-1/\sqrt{J}$. This implies that 
\[
|\Delta_{\Xi}(f,x;q,a) - \Delta_{S}(f,x;q,a)| \leq \frac 1{\varphi(q)} \sum_{\chi\in D} |S_f(x,\chi)|  \ll \frac 1{\varphi(q)} \frac x {(\log x)^{B+o(1)}},
\]
where $D$ is the symmetric difference of the sets $S$ and $\Xi_q$, and the result follows from adding the last two displayed equations.
 \end{proof}

 \begin{corollary} \label{MathResult2Cor**} Fix an integer $J\geq 2$, then $0<B<1-1/\sqrt{J}$ and let $w=(\log x)^{2B}$.    
 For any  $f\in \mathcal C$    we have
\[ 
 \frac 1{\log w}  \sum_{\substack{q\leq x \\ P(q)\leq w}} \max_{(a,q)=1} 
|\Delta_{\Xi(B,\log x)}(f,x;q,a)| \ll    \frac x {(\log x)^{B+o(1)}}.
\]
Moreover
\[ 
 \frac 1{\log w}  \sum_{\substack{q\leq x \\ P(q)\leq w}} \max_{(a,q)=1} 
|\Delta_{J-1}(f,x;q,a)| \ll    \frac x {(\log x)^{B + o(1)}}.
\]
 \end{corollary}
 
Here $P(q)$ denotes the largest prime factor of $q$. 
 
\begin{proof}  Let $\Xi= \Xi(B,\log x)$, which has no more than $J$ elements by Corollary \ref{No Exceptions}(b). 
We begin by bounding the contributions of the values of $q>R:=\exp((\log\log x)^2)$:
\[
\begin{split}
\sum_{\substack{R<q\leq x \\ P(q)\leq w}} \max_{(a,q)=1}  |\Delta_{\Xi}(f,x;q,a)| & \ll
\sum_{\substack{R<q\leq x \\ P(q)\leq w}}  \frac xq +
\sum_{\chi \in \Xi} \sum_{\substack{R<q\leq x \\ P(q)\leq w\\ r_\chi|q}}   \frac xq \\
&\ll   x   \sum_{\substack{q > R\\ P(q)\leq w }}   \frac 1q + x \sum_{\chi \in \Xi} \frac{1}{r_{\chi}} \sum_{\substack{R/r_{\chi} < n \leq x/r_{\chi} \\ P(n) \leq w}} \frac{1}{n},
\end{split}
\]
writing $q=nr_\chi$ in the second sum. The first term, and the contribution to the second term from those $\chi$ with $r_{\chi} \leq \sqrt{R}$ are both acceptable, by the estimates \eqref{tail sum} for smooth numbers. The contribution to the second term from those $\chi$ with $r_{\chi} > \sqrt{R}$ is
\[ \ll \frac{x  }{\sqrt{R}} \sum_{\substack{n \leq x\\ P(n)\leq w}} \frac{1}{n} \ll \frac{x \log w}{\sqrt{R}} \ll \frac{x}{(\log x)^B}, \]
which is also acceptable.

Finally, by Corollary \ref{No Exceptions}(b) we have that  
\[ 
 \frac 1{\log w}  \sum_{\substack{q\leq R \\ P(q)\leq w}} \max_{(a,q)=1} 
|\Delta_{\Xi}(f,x;q,a)| \ll  \frac 1{\log w}  \sum_{q: P(q) \leq w} \frac 1{\varphi(q)} \frac x {(\log x)^{B+o(1)}},
\]
and this is $\ll x/(\log x)^{B+o(1)}$ as $ \sum_{q: P(q)\leq w} 1/{\varphi(q)}\ll \log w$.
This completes the proof of the first part of the Corollary.

Now $\Xi= \{ \psi_1,\ldots, \psi_{k}\}$ for some $k<J$, by definition.   Therefore 
\[
  |\Delta_{J-1}(f,x;q,a)| \leq  |\Delta_{\Xi }(f,x;q,a)| + \frac 1{\varphi(q)}\sum_{\substack{k<j<J \\ r_j|q}} |S_f(x,\chi_j)| , 
\]
and so the result follows from summing this over the $w$-smooth moduli $q$, using   the last part of Lemma  \ref{Lem:Chars} for each $j$, along with the definition of $\Xi$. 
 \end{proof}

\subsection{Making use of the Siegel-Walfisz criterion}
\begin{proposition} \label{Using SW}  Let $f\in \mathcal C$ and $Q \leq x$. For each $q \sim Q$ let $a_q\pmod q$ be a residue class with $(a_q,q) = 1$. Suppose that $\Xi$ is a set of primitive characters, containing $\ll   (\log x)^{C}$ elements,
such that 
 \[
 \sum_{q \sim Q}   \left|  \Delta_{\Xi}(f,x;q,a_q)  \right|  \ll \frac x{(\log x)^B} .
 \]
If the $D$-Siegel-Walfisz criterion holds for $f$, where $D\geq B+C$, then  
\[ 
\sum_{q \sim Q} \left|  \Delta(f,x;q,a_q)  \right|    \ll \frac x{(\log x)^B}. 
\]
\end{proposition}

  \begin{proof}     By definition we have
   \[
  | \Delta(f,x;q,a_q) |  \leq  | \Delta_{\Xi}(f,x;q,a_q) | + \frac 1{\varphi(q)}  \sum_{ \substack{\chi \pmod {q} \\ \chi \in  \Xi_{q},\ \chi\ne \chi_0 }} \left|  S_f(x,\chi) \right| .
\]
Summing this up over $q\sim Q$, and using the hypothesis, we deduce that
\[
 \sum_{\substack{ q\sim Q  }}   | \Delta(f,x;q,a_q) | 
  \leq 
 \sum_{\substack{ \psi\in \Xi \\ \psi\ne 1}}      \sum_{\substack{r_\psi|q\sim Q  \\ \chi \pmod q  \ \text{induced by } \psi } } \frac{|S_f(x,\chi)|}{\varphi(q)}  +O\left(  \frac{x} {(\log x)^{B}} \right).
\]
The third part of Lemma \ref{Lem:Chars} then implies that this is
\[
  \leq 
x  \sum_{\substack{ \psi\in \Xi \\ \psi\ne 1}} \frac{1}{\varphi(r)} \left( \sigma_f(x, \psi) + (\log x)^{-D} \right)   +O\left(  \frac{x} {(\log x)^{B}} \right).
\] 
  The $D$-Siegel-Walfisz criterion implies that for any non-principal $\psi\pmod{r}$, we have
\[
 \frac 1{\varphi(r)} S_f(X,\psi) =  \frac 1{\varphi(r)} \sum_{a \pmod r} \overline{\psi}(a) \Delta(f,X;r,a) \ll    \frac X{(\log x)^{D}} ,
\]
for $x^{1/2}<X\leq 2x$, and so $\sigma_f(x,\psi)/\varphi(r)\ll 1/(\log x)^{D}$. Therefore the above is 
\[
  \ll \# \Xi \cdot \frac x{(\log x)^{D}}  +   \frac x{(\log x)^B}  \ll  \frac x{(\log x)^B} 
  \]
provided $D\geq B+C$.   
\end{proof}

\subsection{Lower bounds} \label{ClaimJusitify}

If $\chi\not\in \Xi_q$ then 
\[
 S_f(x,\chi) =    \sum_{a \pmod q} \overline{\chi}(a) \Delta_\Xi(f,x;q,a)  ,
\]
and so
\[
\frac{ |S_f(x,\chi) |}{\varphi(q)} \leq  \max_{(a,q)=1} |\Delta_\Xi(f,x;q,a) |.
\]
Therefore
\[
\max_{(a,q)=1} |\Delta_\Xi(f,x;q,a) | \geq \max_{\chi\not\in \Xi_q} \frac{ |S_f(x,\chi) |}{\varphi(q)} ;
\]
in particular we deduce that for any primitive $\psi\not\in \Xi$ we have
\[
\sum_{q\sim Q} \max_{(a,q)=1} |\Delta_\Xi(f,x;q,a) | \geq \sum_{\substack{q\sim Q \\ \chi \text{ induced by } \psi}}    \frac{ |S_f(x,\chi) |}{\varphi(q)} .
\]
From the identity 
\[
S_f(x,\chi) = \sum_{m\leq x} h(m) S_f(x/m,\psi)
\]
(see the proof of Lemma~\ref{Lem:Chars})  we might expect that if $|S_f(x ,\psi)|$ is large then each $|S_f(x,\chi)|$ should be too, though this is difficult to prove for every induced $\chi$. However we can do so when the smallest prime factor of $q$ that does not divide $r$ is $>L \log x$, where 
$L:=x/|S_f(x,\psi)|$. Taking absolute values for such $\chi$, and remembering the support of $h(.)$, we have
\[
\begin{split}
|S_f(x,\chi) | &\geq |S_f(x ,\psi) | - \sum_{\substack{m> 1 \\ p|m\implies p|q \\ (m,r)=1}} |S_f(x/m,\psi)| \geq 
\frac xL - \sum_{\substack{m> 1 \\ p|m\implies p|q \\ (m,r)=1}} \frac xm\\
&\geq 
\frac xL - x \left( \prod_{p|q,\ p\nmid r} \left(1 - \frac 1p\right)^{-1} -1 \right)  \sim \frac xL= |S_f(x,\psi)|,\\
\end{split}
\]
since $q$ has $o(\log x)$ prime factors,
For such $q$  we also have $\varphi(q)\sim \varphi(r)q/r$. Therefore, if $|S_f(x,\chi)|$ is significantly larger than $(x\log x)/Q$ then
\[
\sum_{\substack{q\sim Q \\ \chi \text{ induced by } \psi}}    \frac{ |S_f(x,\chi) |}{\varphi(q)} \gg 
 \frac{ |S_f(x,\psi) |}{\varphi(r)}     \sum_{\substack{q\sim Q \\ q=rn \\ p|n\implies p>L \log x}}    \frac{ 1}{n} \gg 
 \frac{ |S_f(x,\psi) |}{\varphi(r)} \cdot \frac 1{\log (L\log x)} .
\]
Therefore if $|S_f(x,\psi)|\gg x/(\log x)^A$ for some  primitive $\psi$ then 
\begin{equation} \label{eq: 3.lbs}
\sum_{q\sim Q} \max_{(a,q)=1} |\Delta_\Xi(f,x;q,a) | \gg \frac{ |S_f(x,\psi) |}{\varphi(r)} \cdot \frac 1{\log\log x} .
\end{equation}
At worst, when all of the $|S_f(x,\psi)|$ with $\psi\not\in \Xi$ are small, we might expect (by orthogonality) that one has $|S_f(x,\psi) |\gg S_{|f|^2}(x,1) \gg \sqrt{x}$ for some primitive character $\psi \pmod r$, from which 
one deduces that 
\[
\sum_{q\sim Q} \max_{(a,q)=1} |\Delta_\Xi(f,x;q,a) | \gg \frac 1 {\log x}  \max_{\substack{\psi \pmod r \\ \psi \ \text{primitive} \\ \psi\not\in \Xi}} \frac{ |S_f(x,\psi) |}{\varphi(r)} .
\]

\section{Formulating the key technical result}
 If $\chi \pmod r$ is in $\Xi$, we write $r=r_\chi$, and note that    it induces a character mod $q$ if and only if $r$ divides $q$, and then the induced character is $\chi\xi_q$ where $\xi_q$ is the principal character mod $q$.
We have the upper bound $|S_f(x,\chi)|\leq \sum_{n\leq x,\ (n,q)=1} 1\ll (\varphi(q)/q) x$ for $x\geq q$. Therefore
$|\Delta_\Xi(f,x;q,a)|\ll \frac{1+|\Xi_q|}q x$ for $x\geq q$. Moreover 
\[
\Delta_\Xi(f,x;q,a) = \frac 1{\varphi(q)} \sum_{ \substack{\chi \pmod q \\ \chi \not\in   \Xi_q }} \chi(a) S_f(x,\chi).
\]

 \begin{corollary} \label{MathResult2Cor} Fix $\varepsilon>0$, let $Q \leq x^{1/2-\varepsilon}$, and let $w \geq 2$. Let $\Xi$ be a set of primitive characters, each with $w$-smooth conductors $\leq  Q/\exp((\log w)^2)$, such that
 \begin{equation} \label{CharHyp}
\sum_{\substack{\chi   \in \Xi}} \frac 1{r_\chi} \ll w^{1/2}.
 \end{equation}
 For any $1$-bounded multiplicative function $f$, we have
\[ 
\begin{split}
& \sum_{q \sim Q}   \max_{(a,q )=1}  \left| \Delta_{\Xi}(f,x;q,a) \right| \\
& \leq 
 \frac 1{\log w}  \sum_{\substack{q_s\leq 2Q \\ P(q_s)\leq w}} \max_{(a,q_s)=1} 
|\Delta_{\Xi}(f,x;q_s,a)|
 +O\left( \frac x{w^{1/2}} + \frac{x \log\log x}{\log x} \right) .
 \end{split}
 \]
 \end{corollary}

The  proof can be modified to allow any $Q  \leq x^{1/2}/2$, though we would need to replace the $(\log\log x)/\log x$ by 
$(\log\log x)/\log (x/Q^2)$ on the right-hand side.

 As we will justify below, Corollary  \ref{MathResult2Cor} is a consequence of:

\begin{theorem} \label{MathResult2} Fix $\varepsilon>0$, let $Q \leq x^{1/2-\varepsilon}$, and let $w \geq 2$.
For any $1$-bounded multiplicative function $f$, we have
 \[ 
 \sum_{q \sim Q} \max_{(a,q)=1} \left|\sum_{\substack{n \leq x\\ n \equiv a\pmod{q}}}  f(n) - \frac 1{q_r}\sum_{\substack{n \leq x\\ n \equiv a\pmod{q_s}}} f(n)\right| \ll \frac x{w^{1/2}} + \frac{x \log\log x}{\log x},
 \]
 where $q_s,q_r$ are the $w$-smooth and the $w$-rough parts of $q$, respectively.
\end{theorem}

This will be proved in Section~\ref{BilinBd}.

 \begin{proof}  [Deduction of Corollary~\ref{MathResult2Cor} from Theorem~\ref{MathResult2}] 
Note that if $\psi\pmod r \in \Xi$ induces a character in $\Xi_q$, then it also induces a character in $\Xi_{q_s}$, since $r$ is $w$-smooth. It follows that
\[ \Delta_{\Xi}(f,x;q,a) = \sum_{\substack{n \leq x \\ n \equiv a\pmod{q}}} f(n) - \frac{1}{\varphi(q)} \sum_{\psi \in \Xi_{q_s}} \psi(a) S_f(x, \psi\xi_q), \]
where $\xi_q$ is the principal character mod $q$. Combining this with the definition of $\Delta_{\Xi}(f,x;q_s,a)$ we arrive at the identity 
 \[ 
\begin{split}
\Delta_{\Xi}(f,x;q,a) &= \frac 1{q_r}\Delta_{\Xi}(f,x;q_s,a)  + \sum_{\substack{n \leq x\\ n \equiv a\pmod{q}}}  f(n) - \frac 1{q_r}\sum_{\substack{n \leq x\\ n \equiv a\pmod{q_s}}} f(n) \\
+   \sum_{\psi\in \Xi_{q_s}}  &\psi(a)   \left(   \frac{S_f(x,\psi)-S_f(x,\psi\xi_q)}{ \varphi(q)} - \left(  \frac{1}{ \varphi(q)} -  \frac{1}{q_r\varphi(q_s)}  \right)S_f(x,\psi)  \right).
\end{split}
\]
We sum the absolute value of  this up over each $q\sim Q$ with $a=a_q$ which maximizes $|\Delta_{\Xi}(f,x;q,a)|$. The first term on the right-hand side is then, summing up over $q=q_rq_s\sim Q$,
\[
\begin{split}
& \leq     \sum_{\substack{q_s\leq 2Q \\ P(q_s)\leq w}} \max_{(a,q_s)=1} 
|\Delta_{\Xi}(f,x;q_s,a)| \sum_{\substack{q_r\sim Q/q_s\\ p|q_r\implies p>w}} \frac 1{q_r}\\
&\ll \frac 1{\log w}  \sum_{\substack{q_s\leq 2Q \\ P(q_s)\leq w}} \max_{(a,q_s)=1} 
|\Delta_{\Xi}(f,x;q_s,a)| +  \sum_{\substack{q_s\sim Q \\ P(q_s)\leq w}} \max_{(a,q_s)=1} 
|\Delta_{\Xi}(f,x;q_s,a)| .
\end{split}
 \]
The last term comes from those $q$ with $q_r=1$, in which case each $|\Delta_{\Xi}(f,x;q_s,a)|\ll(1+|\Xi_q|)\cdot x/ q_s$, 
and so
\[
\begin{split}
  \sum_{\substack{q_s\sim Q \\ P(q_s)\leq w}} & \max_{(a,q_s)=1} |\Delta_{\Xi}(f,x;q_s,a)|
   \ll 
\sum_{\substack{q_s\sim Q \\ P(q_s)\leq w}} \frac{x}{q_s}+ \sum_{\chi \pmod r \in \Xi} \sum_{\substack{q_s\sim Q \\ P(q_s)\leq w\\ r|q_s}}  \frac x{q_s} \\
&    \ll x  v^{-v+o(v)} + \sum_{\substack{\chi \pmod r \in \Xi\\ P(r)\leq w}} \frac xr \sum_{\substack{n\sim Q/r \\ P(n)\leq w }}  \frac 1{n} \ll \left( 1 + \sum_{\substack{\chi   \in \Xi}} \frac 1{r_\chi} \right) xv^{-v+o(v)} 
\end{split}
\]
 writing $q_s=rn$ in the second sum, as $Q/r\geq  \exp((\log w)^2)=w^v$, say,
and this term is $ \ll  x/{w^{1/2}} $ by \eqref{CharHyp}.

By Theorem \ref{MathResult2} we have
\[ 
\sum_{q \sim Q} \max_{(a,q)=1} \left|\sum_{\substack{n \leq x\\ n \equiv a\pmod{q}}}  f(n) - \frac 1{q_r}\sum_{\substack{n \leq x\\ n \equiv a\pmod{q_s}}} f(n)\right| \ll \frac x{w^{1/2}} + \frac{x \log\log x}{\log x}.
\]
  
 Now since $|S_f(x,\psi)|\leq \sum_{n\leq x,\ (n,q_s)=1} 1 \ll (\varphi(q_s)/q_s)x$, we have  
\[
 \left(  \frac{1}{ \varphi(q)} -  \frac{1}{q_r\varphi(q_s)}  \right)S_f(x,\psi)   \ll  \left( \frac {q_r}{\varphi(q_r) } - 1\right)  \frac 1{q_r\varphi(q_s)} \cdot \frac  {\varphi(q_s)}{q_s} x=    \left( \frac {q_r}{\varphi(q_r)} - 1 \right) \frac x{q}.
 \]
 Moreover, writing $n=ab$ where $p|a\implies p|q$ and $(b,q)=1$, we have 
 \[
  \begin{split}
  \left|\frac{S_f(x,\psi)-S_f(x,\psi\xi_q)}{ \varphi(q)}\right| &=  \frac{1}{ \varphi(q)}\left|\sum_{\substack{ab\leq x \\ a>1 }}  (f\overline\psi)(a)(f\overline\psi)(b)\right| \leq   \frac{1}{ \varphi(q)}\sum_{\substack{1<a \leq x  \\ p|a \implies p|q_r}}    \sum_{ \substack{b\leq x/a \\ (b,q)=1}}  1  \\
  &\ll   \frac{1}{ \varphi(q)}\sum_{\substack{1<a \leq x/q  \\ p|a \implies p|q_r}}    \frac{\varphi(q)}{q} \frac xa+ \frac{1}{ \varphi(q)}\sum_{\substack{  x/q<a\leq x  \\ p|a \implies p|q_r}}    \frac xa \\
   &\ll   \left( \frac {q_r}{\varphi(q_r)} - 1 \right) \frac x{q} + x^{o(1)}  \\
   \end{split}
 \]
since there are $x^{o(1)}$  integers $a$ in an interval $[X,2X]$ with $X\geq x/q$, all of whose prime factors divide $q_r$ (to see this note that the worst case is when $q_r$ is the product of all the primes $\ll \log x$, and then this easily follows from our estimates for smooth numbers).

Therefore
\[
\begin{split}
& \sum_{q\sim Q} \left|  \sum_{\psi\in \Xi_{q_s}}  \psi(a)   \left(   \frac{S_f(x,\psi)-S_f(x,\psi\xi_q)}{ \varphi(q)} - \left(  \frac{1}{ \varphi(q)} -  \frac{1}{q_r\varphi(q_s)}  \right)S_f(x,\psi)  \right)  \right| \\
& \ll
\sum_{\chi\in \Xi} \sum_{\substack{q\sim Q \\ r_\chi|q}}    \left( \frac {q_r}{\varphi(q_r)} - 1 \right) \frac x{q} + \sum_{\chi \in \Xi} \sum_{\substack{q \sim Q \\ r_{\chi}|q}} x^{o(1)} \\
& = \sum_{\chi\in \Xi} \sum_{\substack{ q_s\leq 2Q \\ r_\chi|q_s}} \frac 1{q_s} 
\sum_{\substack{q_r\sim Q/q_s }}    \left( \frac {q_r}{\varphi(q_r)} - 1 \right) \frac x{q_r} + \sum_{\chi\in \Xi} \frac{Qx^{o(1)} }{r_\chi}\\
& \ll \sum_{\chi\in \Xi} \frac{\log w}{ r_\chi} \cdot
\frac x{w\log w} + \sum_{\chi\in \Xi} \frac{Qx^{o(1)} }{r_\chi} \ll \frac x{w} \cdot  \sum_{\chi\in \Xi} \frac{1}{ r_\chi}  \ll \frac x{w^{1/2}} ,
\end{split}
\]
 since  $\sum_{\substack{ n\leq Q/r \\ P(n)\leq w}} 1/n\ll \log w $ where $q_s=r_\chi n$, and for $Q_s=Q/q_s$  we have
 \[ 
 \begin{split}
  \sum_{q_r \sim Q_s} & \left( \frac {q_r}{\varphi(q_r)} - 1 \right) \frac x{q_r}    = \sum_{q_r \sim Q_s} \frac x{q_r} \sum_{\substack{d>1 \\ d|q_r}} \frac{\mu^2(d)}{\varphi(d)} 
  = \sum_{\substack{ 1<d\leq 2Q_s\\ p|d\implies p>w}}  \frac{\mu^2(d)}{\varphi(d)}  \sum_{\substack{q_r \sim Q _s\\ d|q_r}} \frac x{q_r} \\ &\ll
x \sum_{\substack{ 1<d\leq 2Q_s\\ p|d\implies p>w}}  \frac{\mu^2(d)}{d\varphi(d)}  
 \leq x \left(   \prod_{p>w}   \left(  1+\frac 1{p(p-1)}  \right)  -1 \right) \ll \frac x{w\log w}.
 \end{split}
 \]
Collecting up the estimates above the result follows.
 \end{proof}

\subsection{Using Ramar\'{e}'s weights} 

Let $a_q \pmod q$ be the arithmetic progression with $(a,q)=1$ for which 
 \[  \left|\sum_{\substack{n \leq x \\ n \equiv a\pmod{q}}} f(n) - \frac{1}{q_r} \sum_{\substack{n \leq x \\ n\equiv a\pmod{q_s}}} f(n)\right| \]
 is maximized, and then select $\xi_q$ with $|\xi_q|=1$ so that this equals
 \[ \xi_q \left( \sum_{\substack{n \leq x \\ n \equiv a\pmod{q}}} f(n) - \frac{1}{q_r} \sum_{\substack{n \leq x \\ n\equiv a\pmod{q_s}}} f(n) \right) . \]
 Therefore the left-hand side of the equation in Theorem \ref{MathResult2} can be re-written as 
 \[  \sum_{n \leq x} f(n) F(n) \]
 where, here and throughout, the function $F$ is defined by
\begin{equation} \label{F-value}
 F(n) := \sum_{q\sim Q} \xi_q \left(\1_{n \equiv a_q\pmod{q}} - \frac{1}{q_r} \1_{n\equiv a_q\pmod{q_s}}\right). 
 \end{equation}

Sums like 
\[ \sum_{n \leq x} f(n) F(n)  \]
have been tackled in the literature using the Cauchy-Schwarz inequality and then studying bilinear sums, such  as what happens after  (17) in \cite{MV}. Here we develop a formal inequality, which is a variant of  Proposition 2.2 of \cite{Gre}.   The use of Ramar\'{e}'s identity in this context was pioneered by Matom\"{a}ki and Radziw\l\l~\cite{MR}.

\begin{proposition}\label{prop:ramare}
Let $F: \Z \to \C$ be an arbitrary function. Let $2 \leq Y < Z < x^{1/9}$ be parameters,  and write $u = (\log Z)/\log Y$. Then for any $1$-bounded multiplicative function $f$, we have
\[ \sum_{Z^9 \leq n \leq x} f(n) F(n) \ll \frac{T}{Y \log Y} + E_{\text{sieve}} + E_{\text{bilinear}}, \]
where
\[ T = \max_{d \leq Z^2} d \sum_{\substack{n \leq x \\ d\mid n}} |F(n)|, \]
\[ E_{\text{sieve}} = \sum_{n \leq x} |f(n)F(n)| \1_{(n,\prod_{Y \leq p < Z}p) = 1}, \]
and
\[ E_{\text{bilinear}} = \max_{P \in [Y,Z)} \left(Px \cdot \E_{p,p' \sim P} \left| \sum_{m \leq \min(x/p, x/p')} F(pm) \overline{F(p'm)} \right|\right)^{1/2}. \]
\end{proposition}

\begin{proof} We use  Ramar{\'e}'s weight function:
\[ w(n) = \frac{1}{\#\{Y \leq p < Z: p \mid n\} + 1},\]
to obtain the identity
\[ 
\sum_{\substack{Y \leq p < Z \\ p^k\|  n,\ k\geq 1}} w(n/p^k) = \begin{cases} 1 & \text{if }p \mid n\text{ for some }Y \leq p < Z \\ 0 & \text{otherwise.} \end{cases} 
\]
Therefore, letting $F(n)=0$ for $n<Z^9$,
\[
\begin{split}
\left|\sum_{n \leq x} f(n) F(n)\right| &\leq E_{\text{sieve}} + \Big|\sum_{\substack{n \leq x }} f(n) F(n) \sum_{\substack{Y \leq p < Z \\ p^k\|n}} w(n/p^k)\Big|\\
&\leq E_{\text{sieve}} + \Big|\sum_{\substack{n=mp^k \leq x\\ Y \leq p < Z \\   (p,m)=1}} w(m)f(m) \cdot  f(p^k) F(mp^k)\Big|,
\end{split}
 \]
writing $n=mp^k$. We replace each term where $p^2|n$, currently written as $n=mp^k$, by $n$ written as $m'p^2$.  Given that $|w(.)|,|f(.)|\leq 1$ the difference is
\[
\leq 2\sum_{Y \leq p < Z} \sum_{\substack{n \leq x \\ p^2 \mid n}} |F(n)| \leq 2T \sum_{Y \leq p < Z} \frac{1}{p^2} \ll \frac{T}{Y\log Y}.
\]
Therefore,
\begin{equation}\label{eq:Sigma} 
\left|\sum_{n \leq x} f(n) F(n)\right| \leq |\Sigma'|  + E_{\text{sieve}} + O\left(   \frac{T}{Y\log Y}    \right)
\end{equation}
where
\[ 
\Sigma' := \sum_{m \leq x/Y} w(m) f(m) \sum_{\substack{Y \leq p < Z \\ p \leq x/m}} f(p) F(pm). 
\]
We  divide the range $[Y,Z)$ for $p$, dyadically, into
\[ \Sigma'(P) := \sum_{m \leq x/P} w(m) f(m) \sum_{\substack{P \leq p < 2P \\ p \leq x/m}} f(p) F(pm) \]
for $P \in [Y, Z)$. Since $F(pm)$ is supported on $pm \geq Z^9$, we may add the restriction $m \geq Z^8$ to the sum. By Lemma 2.1 of~\cite{Gre} we have
\[ \sum_{Z^8<m \leq x/P} w(m)^2 \ll \frac{x}{P (\log u)^2}. \]
Therefore, by the Cauchy-Schwarz inequality, we obtain
\[ |\Sigma'(P)|^2 \ll \frac{x}{P(\log u)^2} \sum_{m \leq x/P} \left|\sum_{\substack{P \leq p < 2P \\ p \leq x/m}} f(p) F(pm) \right|^2. \]
After expanding the square and changing the order of summation, we obtain
\begin{align*} 
|\Sigma'(P)|^2 & \ll \frac{x}{P(\log u)^2} \sum_{P \leq p,p' < 2P} f(p) \overline{f(p')} \sum_{m \leq \min(x/p, x/p')} F(pm) \overline{F(p'm)} \\
& \ll \frac{x}{P(\log u)^2} \sum_{P \leq p,p' < 2P} \left| \sum_{m \leq \min(x/p, x/p')} F(pm) \overline{F(p'm)} \right| \\
& \ll \frac{x}{P (\log u)^2} \cdot \frac{\pi(P)^2}{Px} E_{\text{bilinear}}^2 = \left(\frac{E_{\text{bilinear}}}{(\log P)(\log u)}\right)^2.
\end{align*}
Taking square root and then summing dyadically over $P$, we obtain $\Sigma' \ll E_{\text{bilinear}}$. This completes the proof.
\end{proof}

\subsection{Bounding the sieve term and the bilinear term}  \label{BilinBd}
We now apply Proposition~\ref{prop:ramare} to the function $F$ in \eqref{F-value} to prove Theorem~\ref{MathResult2}. Let $2 \leq Y < Z < x^{1/9}$ be two parameters. Recall the notation that $u = (\log Z)/\log Y$.

\begin{lemma}[Sieve term] \label{sievebound}
We have
\[ \sum_{n \leq x} |F(n)| \1_{(n, \prod_{Y\leq p < Z}p)=1} \ll \frac{x}{u}. \]
\end{lemma}

\begin{proof}
Using the trivial bound
\[ |F(n)| \leq \sum_{q\sim Q} \left(\1_{n\equiv a_q\pmod{q}} + \frac{1}{q_r} \1_{n\equiv a_q\pmod{q_s}}\right), \]
we may bound the desired expression by
\[ \sum_{q\sim Q} \sum_{\substack{n \leq x \\ n \equiv a_q\pmod{q}}} \1_{(n,\prod_{Y \leq p < Z}p)=1} + \sum_{q\sim Q} \frac{1}{q_r} \sum_{\substack{n \leq x \\ n \equiv a_q\pmod{q_s}}} \1_{(n,\prod_{Y \leq p < Z}p)=1}. \]
By an upper bound sieve, the inner sum  over $n$ in the first  term is
\[ \ll \frac{x}{q} \prod_{\substack{Y \leq p < Z \\ p\nmid q}} \left(1 - \frac{1}{p}\right) \ll \frac{x}{u} \cdot \frac{1}{\varphi(q)}  \]
for any $q \sim Q$. Thus the first term is $O(x/u)$ since
\[ \sum_{q \sim Q} \frac{1}{\varphi(q)} \ll 1. \]
The second term is dealt with similarly: bound the inner sum over $n$ by $O(x/u\varphi(q_s))$ and then the second term is
\[ \ll \frac{x}{u} \sum_{q \sim Q} \frac{1}{q_r\varphi(q_s)} \leq \frac{x}{u} \sum_{q\sim Q} \frac{1}{\varphi(q)} \ll \frac{x}{u}. \]
This completes the proof of the lemma.
\end{proof}

\begin{lemma}[Bilinear term] \label{BilinBound}
For any $P,Q \geq 2$ we have
\[ \E_{p,p'\sim  P} \left| \sum_{m \leq \min(x/p, x/p')} F(pm) \overline{F(p'm)} \right| \ll \frac{x}{P}\left(\frac{1}{w\log w} + \frac{P^{0.1} + \log x}{\pi(P)}\right) +Q^2 , \]
where $p,p'$ denote primes, where, as usual, $w\leq x$.
\end{lemma}

\begin{proof}
By the definition of $F$, we change the order of summation to write the inner sum over $m$ as
\[ \Sigma(p,p') := \sum_{q,q'\sim Q} \xi_q \overline{\xi_{q'}} \sum_{m \leq \min(x/p, x/p')} K(p,p',q,q';m), \]
where $K(p,p',q,q';m)$ is the expression
\[\left(\1_{pm\equiv a_q\pmod{q}} - \frac{1}{q_r} \1_{pm\equiv a_q\pmod{q_s}}\right) \left(\1_{p'm\equiv a_{q'}\pmod{q'}} - \frac{1}{q_r'} \1_{p'm\equiv a_{q'}\pmod{q_s'}}\right). \]
The inner sum $\sum_m K(p,p',q,q';m)$ can be written as a sum of four sums, the first of which is
\[ \sum_{m \leq \min(x/p, x/p')} \1_{pm\equiv a_q\pmod{q}} \cdot \1_{p'm\equiv a_q'\pmod{q'}}. \]
This sum should have a main term of
\[ S(p,p',q,q') := \begin{cases} \frac{\min(x/p, x/p')}{[q,q']} & \text{if }(p,q) = (p',q') = 1\text{ and }p'a_q \equiv pa_q' \pmod{(q,q')}, \\ 0 & \text{otherwise,} \end{cases} \]
with an error of $O(1)$. Similarly for the other three sums. It follows that the sum 
\[ \sum_{m \leq \min(x/p, x/p')} K(p,p',q,q';m) = g_0(p,p',q,q') + O(1), \]
where the main term $g_0(p,p',q,q')$ is defined by
\[  g_0(p,p',q,q') = S(p,p',q,q') - \frac{1}{q_r}S(p,p',q_s,q') - \frac{1}{q_r'} S(p,p',q,q_s') + \frac{1}{q_rq_r'} S(p,p',q_s,q_s'). \]
The total contribution from the $O(1)$ error is $O(Q^2)$, which is acceptable. Thus it suffices to show that
\[ \E_{p,p' \sim P} \sum_{q,q' \sim Q} |g_0(p,p',q,q')| \ll \frac{x}{P}\left(\frac{1}{w\log w} + \frac{P^{0.1}+\log x}{\pi(P)}\right). \]
Note that we have the upper bound
\begin{equation}\label{eq:g0bound1} 
|g_0(p,p',q,q')| \ll \frac{x}{P} \cdot \frac{(q,q')}{qq'}.
\end{equation}
Moreover, when $S(p,p',q,q')=0$ we have the (possibly) improved upper bound
\begin{equation}\label{eq:g0bound2} 
|g_0(p,p',q,q')| \ll \frac{x}{P} \cdot \frac{(q_s,q_s')}{qq'}. 
\end{equation}

\subsection*{Case 0}

First consider the case when $p \mid q$ or $p' \mid q'$. Then $S(p,p',q,q') = 0$ and, so by~\eqref{eq:g0bound2},
\[ 
\E_{p,p' \sim P} \sum_{\substack{q,q' \sim Q \\  p \mid q \text{ or } p' \mid q}} |g_0(p,p',q,q')| \ll 
\frac{x}{P\pi(P)^2} \sum_{p,p' \sim P} \sum_{\substack{q,q' \sim Q \\  p \mid q \text{ or } p' \mid q}} \frac{(q_s,q_s')}{qq'},
 \]
 which by symmetry is
\[
\ll 
\frac{x}{P\pi(P)} \sum_{p\sim P} \sum_{\substack{q,q' \sim Q \\  p \mid q }} \sum_{\substack{ d\geq 1 \\ P(d)\leq w \\ d|q, d|q'}} \frac{d}{qq'} = \frac{x}{P\pi(P)} \sum_{p\sim P} \frac 1p \sum_{\substack{ d\geq 1 \\ P(d)\leq w }}  \frac 1d
\sum_{\substack{r \sim Q/pd \\ r'\sim Q/d }}\frac{1}{rr'}
 \]
 writing $q=pdr, q'=dr'$, which is
 \[
 \ll \frac{x}{P\pi(P)} \cdot \frac{1}{\log P} \cdot \log w \cdot 1^2 \ll \frac{x\log w}{P^2},
 \]
which is easily acceptable.

\subsection*{Case 1} 

Now consider the case when $(p,q) = (p',q') = 1$ and $p'a_q \equiv pa_q' \pmod{(q,q')}$. Since $(q_s,q') = (q,q_s') = (q_s,q_s')$, we have
\[ g_0(p,p',q,q') = \min(x/p, x/p') \left(\frac{(q,q')}{qq'} - \frac{(q_s,q_s')}{qq'}\right). \]
It vanishes unless $(q_r,q_r') > 1$, and thus by~\eqref{eq:g0bound1} it suffices to show that
\[
\Sigma_1:= \frac{1}{\pi(P)^2}  \sum_{\substack{ q,q'\sim Q \\ (q_r,q_r')>1}}  \sum_{\substack{ p,p'\sim  P \\ (p,q) = (p',q') = 1 \\ p'a_q \equiv pa_q'\pmod{(q,q')} }} 
\frac{(q,q')}{qq'} \ll \frac{1}{w\log w} + \frac{P^{0.1} + \log x}{\pi(P)}.
\]
For fixed $q,q',p$, the constraint $p'a_q \equiv pa_{q'} \pmod{(q,q')}$ imposes a congruence condition on $p' \pmod{(q,q')}$, and  the number of $p'$ satisfying it is
\[ \ll \begin{cases} \pi(P)/\varphi((q,q')) & \text{if }(q,q') \leq P^{0.9}, \\ P/(q,q') + 1  & \text{if }(q,q') > P^{0.9}. \end{cases} \] 
Here the bound in the first case follows from Brun-Titchmarsh, and  in the second case by dropping the primality condition on $p'$.  Therefore
\begin{equation} \label{Bound1}
\Sigma_1\ll  \frac{1}{Q^2}  \sum_{\substack{ q,q'\sim Q\\  (q_r,q_r')>1}}   \frac{(q,q')}{\varphi((q,q')) }   +
 \frac{1}{\pi(P) Q^2}  \sum_{\substack{ q,q'\sim Q\\ (q,q') > P^{0.9}  }}    (P+ (q,q') ) .
\end{equation} 
In the second term let $d=(q,q')>P^{0.9}$ and then the sum is 
\[ \leq \sum_{P^{0.9} < d \leq 2Q} (P+d) \sum_{\substack{q,q' \sim Q \\ d\mid (q,q')}} 1 \ll \sum_{P^{0.9} < d \leq 2Q} (P+d) \frac{Q^2}{d^2} \ll Q^2(P^{0.1} + \log x), \]
which is acceptable. The first term has the restriction that $(q_r,q_r') > 1$, which implies that there is some prime $p>w$ with $p \mid (q,q')$. Writing $q=pm$ and $q'=pm'$, we have
\[ \sum_{\substack{ q,q'\sim Q\\  (q_r,q_r')>1}}   \frac{(q,q')}{\varphi((q,q')) }  \leq \sum_{w < p \leq 2Q} \frac{p}{p-1} \sum_{m,m' \sim Q/p} \frac{(m,m')}{\varphi((m,m'))}. \]
Using the identity
\[ \frac{(m,m')}{\varphi((m,m'))} = \sum_{d \mid (m,m')} \frac{\mu^2(d)}{\varphi(d)}, \]
we can bound the sum over $q,q'$ by
\[ 
\begin{split} \sum_{w < p \leq 2Q} \frac{p}{p-1} \sum_{d \leq 2Q/p} \frac{\mu^2(d)}{\varphi(d)} \sum_{\substack{m,m' \sim Q/p \\ d\mid (m,m')}} 1 & \ll \sum_{w < p \leq 2Q}  \sum_{d} \frac{\mu^2(d)}{\varphi(d)} \left(\frac{Q}{pd}\right)^2 \\
& \ll Q^2 \sum_{p>w} \frac{1}{p^2} \ll \frac{Q^2}{w\log w}. \end{split}\]
This completes the task of bounding $\Sigma_1$.

\subsection*{Case 2} 
Finally consider the case when $(p,q) = (p',q') = 1$ and $p'a_q \not\equiv pa_q' \pmod{(q,q')}$. If we further have $p'a_q \not\equiv pa_q' \pmod{(q_s,q_s')}$, then 
\[ S(p,p',q,q') = S(p,p',q_s,q') = S(p,p',q,q_s') = S(p,p',q_s,q_s') = 0, \]
and thus $g_0(p,p',q,q') = 0$. Hence we may impose the condition $p'a_q \equiv pa_q' \pmod{(q_s,q_s')}$. By~\eqref{eq:g0bound2} it suffices to show that
\[
\Sigma_2:= \frac{1}{\pi(P)^2}  \sum_{\substack{ q,q'\sim Q  }}  \sum_{\substack{ p,p'\sim  P \\ (p,q) = (p',q') = 1 \\ p'a_q \not\equiv pa_q'\pmod{(q,q')} 
\\  p'a_q\equiv pa_q'\pmod{(q_s,q_s')}}} 
\frac{(q_s,q'_s)}{qq'} \ll \frac{1}{w\log w} + \frac{P^{0.1} + \log x}{\pi(P)}.
\]
Note that the sum is nonempty only  if $(q_r,q_r') > 1$. Arguing as in Case 1, the number of $p'$ satisfying the congruence condition on $p'\pmod{(q_s,q_s')}$ is
\[ \ll \begin{cases} \pi(P)/\varphi((q_s,q_s')) & \text{if }(q_s,q_s') \leq P^{0.9} \\ P/(q_s,q_s') + 1 & \text{if }(q_s,q_s') > P^{0.9}. \end{cases} \]
This leads to the upper bound
\begin{equation*} 
\Sigma_2\ll  \frac{1}{ Q^2}  \sum_{\substack{ q,q'\sim Q\\  (q_r,q_r')>1}}   \frac{(q_s,q_s')}{\varphi((q_s,q_s')) }   +
 \frac{1}{\pi(P) Q^2}  \sum_{\substack{ q,q'\sim Q\\ (q_s,q_s') > P^{0.9}  }}    (P+ (q_s,q_s') ) .
\end{equation*} 
Now as $(q_s,q_s')\leq (q,q')$ and $\frac{(q_s,q_s')}{\varphi((q_s,q_s')) }   \leq \frac{(q,q')}{\varphi((q,q')) } $, we   bound 
$\Sigma_2$ by the same quantity with which we bounded $\Sigma_1$ in \eqref{Bound1}, and the result follows.
\end{proof}


 \subsection{Putting the pieces together} 

We now have the ingredients to deduce Theorem~\ref{MathResult2} from Proposition~\ref{prop:ramare}.

\begin{proof} [Proof of Theorem \ref{MathResult2}] 
Recall that the left-hand side of the equation in Theorem \ref{MathResult2} can be re-written as 
 \[  \sum_{n \leq x} f(n) F(n) \]
 where the function $F$ is defined as in~\eqref{F-value}. We bound this by applying Proposition \ref{prop:ramare}. Set $Y = (\log x)^4$ and $Z = x^{\ee/2}$, so that $u = \log Z/(\log Y) \asymp \log x/(\log\log x)$. By Lemma \ref{sievebound} we have 
\[ E_{\text{sieve}}  \ll \frac{x}{u} \ll \frac{x\log\log x}{\log x}. \]
 By Lemma \ref{BilinBound}, the assumption $Q \leq x^{1/2-\ee}$, and our choice of $Y$ and $Z$, we have
\[  E_{\text{bilinear}}    \ll  \frac x{\log x} + \frac{x}{(w\log w)^{1/2}}  . \]
To bound $T$, provided $d\leq Z^2\leq x/Q$ we have
\[
  \sum_{\substack{n \leq x \\ d\mid n}} |F(n)| \leq  \sum_{q\sim Q}  \sum_{\substack{n \leq x \\ n \equiv a_q\pmod{q}\\ d\mid n}} 1
 +  \frac 1{q_r} \sum_{q\sim Q}  \sum_{\substack{n \leq x \\ n \equiv a_q\pmod{q_s}\\ d\mid n}} 1 \ll  \sum_{q\sim Q}  \frac x{qd} \ll  \frac xd,
\]
so that $T\ll x$. The proof is completed by combining all these estimates together.
\end{proof}

\section{Good error terms for smooth-supported $f$ in arithmetic progressions} 

In the following result we will prove a good estimate for all $f$ supported on $y$-smooth integers. In this article we will only use this with $y$ a fixed power of $x$, but the full range will be useful in the sequel~\cite{DGS}.

\begin{proposition} \label{Errorfaps1} Fix $B\geq 0$ and $0<\eta<\frac 12$.  Given $y=x^{1/u}$ in the range
$ x^{1/2-\eta}\geq y \geq \exp(3(A+1+\epsilon) (\log\log x)^2/(\log\log\log x))$, let 
\[
R=R(x,y):=\min\{ y^{ \frac{  \log\log\log x}{3\log u}},  x^{\frac \eta{3\log\log x}}\} \ (\leq y^{1/3 + o(1)}).
\]   
Suppose that $f\in \mathcal C$, and is only supported on $y$-smooth numbers. 
There exists a set, $\Xi$, of primitive characters $\psi \pmod r$ with $r\leq R$, containing $\ll   u^{u+o(u)}(\log x)^{6B+7+ o(1)}$ elements, such that  if $q\leq R$ and $(a,q)=1$ then
\[    
| \Delta_{\Xi}(f,x;q,a) | \leq  \frac 1{\varphi(q)}   \sum_{ \substack{\chi \pmod q \\ \chi \not\in  \Xi_q }} \left|  S_f(x,\chi) \right| \ll  \frac 1{\varphi(q)}  \frac{\Psi(x,y)} {(\log x)^{B}} .
\]
\end{proposition}

 This immediately implies Proposition \ref{MainCor}.

Fix $\varepsilon>0$. In Proposition \ref{Errorfaps1} we let $\Xi=\Xi(2B+2+\varepsilon)$, where  we define $\Xi(C)$ to be the set of primitive characters $\psi \pmod r$ with $r\leq R$ such that there exists $x^{\eta}< X\leq x$ for which
\begin{equation} \label{Hyp3}
|S_f(X,\psi)|    \geq \frac {\Psi(X,y)}{(u\log u)^4(\log x)^{C}}  .
\end{equation}
We prove that $\Xi(C)$ has $\ll u^{u+o(u)}(\log x)^{3C+1}$ elements in section \ref{LargeS}.

\subsection{A further support restriction for $f$}   We will deduce Proposition \ref{Errorfaps1}   from a similar (but stronger) result restricted to multiplicative functions $f$, which are only supported  on those prime powers $p^k$ for which 
$p\in (qL,y]$ where $L\geq q$ and $L\geq (\log x)^C$ for an appropriate value of $C$.

Fix a real number  $A$. Let $\mathcal T_q^*(A)=\mathcal T_{q,f}(A ,y^2(q^2(\log x)^{2A})^{1/\alpha})$ where $\mathcal T_{q,f}(A,z)$ is the set of characters $\chi \pmod q$ 
for which there exists  $x/z<X\leq x$ such that
\begin{equation} \label{Hyp2}
 |S_f(X,\chi)|     > \frac { \Psi(X,y)}{(u\log u)^4(\log x)^{2A+2+\varepsilon}}  .
\end{equation}

\begin{proposition} \label{Errorfaps1q} Fix $A\geq 0$ and $0<\eta<\frac 12$, and let
\[
x^{1/2-\eta}\geq y \geq \exp(3(A+1) (\log\log x)^2/(\log\log\log x)).
\]
Given $q\leq \min\{ y^{1/2} , x^{ \eta/2} \}/(\log x)^{ A+2}$
 with 
$(a,q)=1$,    let $L=L_q:=q+(\log x)^{A+1}$.  If $f\in \mathcal C$  is only supported on prime powers $p^k$  with $qL<p\leq y$, then
 \begin{equation} \label{smallsupport}  
| \Delta_{\mathcal T_q^*(A)}(f,x;q,a) | \leq  \frac 1{\varphi(q)}   \sum_{ \substack{\chi \pmod q \\ \chi \not\in  \mathcal  T_q^*(A) }} \left|  S_f(x,\chi) \right| \ll  \frac 1{\varphi(q)}  \frac{\Psi(x,y)} {(\log x)^{A}} .
\end{equation}
\end{proposition}

By minor modifications of the proof of  Proposition \ref{Errorfaps1q}  we will deduce the following:

\begin{corollary} \label{ErrorfapsA=0}  
Fix $0 < \eta < \tfrac 12$ and $\varepsilon>0$, and let
\[
x^{1/2-\eta}\geq y \geq \exp(6(\log\log x)^2/(\log\log\log x)).
\]
Given $q\leq \min\{ y^{1/2} , x^{ \eta/2} \}$
 with 
$(a,q)=1$,    let $L=L_q:=q+(\log x)^{2}$.  If $f\in \mathcal C$  is only supported on prime powers $p^k$  with $qL<p\leq y$, then
 \[ 
   \sum_{ \substack{\chi \pmod q }} \left|  S_f(x,\chi) \right| \ll  \Psi(x,y) (\log x)^{3+\varepsilon}     .
\]
\end{corollary}

By Cauchying, the upper bound is $\ll q^{1/2} \Psi(x,y) $, so Corollary \ref{ErrorfapsA=0} is better in a wide range.

In the next seven subsections we will prove Proposition \ref{Errorfaps1q}. 
 Throughout this section we will write $(u\log u)^2=:V$ for convenience, so that if $1\leq d\leq y^2$ then, as a consequence of 
 \eqref{alphaCompare},
\begin{equation} \label{SmoothCompare}
\Psi(x/d,y)\ll V\Psi(x,y)/d.
\end{equation}

\subsection{Harper's identity and Perron's formula}

\begin{proof} [Proof of Proposition \ref{Errorfaps1q}]  We use Harper's identity:\footnote{Proved by Harper in an unpublished precursor to \cite{GHS}.}
\[
f(n)\log n = \Lambda_f(n)   + \int_{\beta=0}^\infty  n^{-\beta} \sum_{abm=n} \Lambda_f(a) a^\beta \Lambda_f(b)  f(m) d\beta .
\]
Multiplying through by $\overline\chi(n)$ and summing over  $n$ we obtain
\[
\begin{split}
\sum_{x<n\leq 2x} f(n)\overline\chi(n) \log n &=  \int_{\beta=0}^\infty   \sum_{\substack{x<abm\leq 2x \\ a,b\leq y}}  (abm)^{-\beta} \Lambda_f(a)\overline\chi(a) a^\beta \Lambda_f(b)\overline\chi(b)  f(m)\overline\chi(m)   d\beta \\
& + \text{Error},
\end{split}
\]
where the error term is $O(y\log x/\log y)$ from the contribution of the $\Lambda_f(n)$, plus
\[
\ll \sum_{\substack{x<abm\leq 2x \\ a \text{ or } b> y\\ P(abm)\leq y}}   \frac{\Lambda(a)   \Lambda(b)}{\log (x/a+1)} \ll
 \sum_{\substack{ ab \leq 2x \\ a \text{ or } b> y\\ P(ab)\leq y}}   \frac{\Lambda(a)   \Lambda(b)}{ \log (x/a+1)} \Psi(2x/ab,y).
 \]
Now if $a=p^i, b=q^j$ then the term with $a$ replaced by $ap$ or the term with $b$ replaced by $bq$ is no more than this term times a constant $<1$ that depends only on $\alpha$ (by \eqref{alphaCompare1} and  \eqref{alphaCompare2}). Therefore we can restrict our attention to where $a$ is a prime power $<y$ and $b$ is a prime power in $(y,y^2]$. Therefore the above is
\[
\begin{split}
&\ll \sum_{a\leq y < b\leq y^2}  \frac{\Lambda(a)   \Lambda(b)}{ \log x} \Psi(x/ab,y)\ll
\frac{V^{3/2} \Psi(x,y)} {\log x} \sum_{a\leq y < b\leq y^2}  \frac{\Lambda(a)   \Lambda(b)}{ ab} \\
 & \ll  \frac{\Psi(x,y)}{\sqrt{y}} \cdot \frac{V^{3/2}\log y}{\log x}
\ll \frac { \Psi(x,y) }{\varphi(q) (\log x)^{A-1}} ,
\end{split}
\]
 by \eqref{SmoothCompare}, as $x^{1/2}>y\geq V^3 \varphi(q)^2 (\log x)^{2A-2}$. 
 
We use Perron's formula to try to work with the main term, applying it at $x$ and $2x$. Wlog we assume that $m$ is only supported on $(x/y^2,2x]$.
Therefore our integrand equals, for $c=1+1/\log x$ and $x \notin \mathbb{Z}$,
\[
\frac 1{2i\pi} \int_{Re(s)=c}   \sum_{\substack{x/y^2<m\leq 2x \\ a,b\leq y}}  (bm)^{-\beta} \Lambda_f(a)\overline\chi(a)  \Lambda_f(b)\overline\chi(b)  f(m)\overline\chi(m)   \left( \frac{x}{abm} \right)^s   \frac{2^s-1} s ds,
\]
which we will truncate at a height $T$ where 
\[ T=\varphi(q)U \ \text{with}  \ U:= V(\log x)^{A+1+\varepsilon} .
\]
The main term of the integrand then equals
\[
\frac 1{2i\pi} \int_{\substack{s=c+it \\ |t|\leq T}}
\sum_{a \leq y}  \frac{ \Lambda_f(a)\overline\chi(a)  } { a^{s}}  \sum_{ b\leq y}   \frac{ \Lambda_f(b)\overline\chi(b) } { b^{\beta+s}} \sum_{x/y^2<m\leq 2x}   \frac{ f(m)\overline\chi(m) } { m^{\beta+s}}
\frac{(2^s-1)x^s} s ds ,
\]
and, taking  absolute values, we therefore have that the main part of our integral is
\[
\ll x \int_{\substack{s=c+it \\ |t|\leq T}}     \int_{\beta=0}^\infty  \left|  \sum_{a \leq y} \frac{ \Lambda_f(a)\overline\chi(a)  } { a^{s}} \right| \left|  \sum_{ b\leq y}   \frac{ \Lambda_f(b)\overline\chi(b) } { b^{\beta+s}} \right| \left| \sum_{x/y^2<m\leq 2x}   \frac{ f(m)\overline\chi(m) } { m^{\beta+s}}   \right| d\beta    \frac{dt }{1+|t|}.
\]

\subsection{Truncation bounds; the contribution with $|t|>T$}
The usual formula for truncation (see Theorems 5.2 and 5.3 of~\cite{MV-book}) means that for the integral for $n$ up to $x$ we have an error term of
\[
\ll    \sum_{\substack{ a,b\leq y \\ x/y^2<m\leq 2x \\ P(m)\leq y}}  (bm)^{-\beta} \Lambda (a)\ \Lambda (b)  \left( \frac{x}{abm} \right)^{c}  \min\left \{ 1 , \frac 1{T|\log  (x/abm)|} \right\} .
\]
Integrating over $\beta$, the $(bm)^{-\beta}$ contributes $1/\log x$. We have $1/2y^2\leq x/abm\leq y^2$ and so $(x/abm)^{c} \asymp x/abm$, so the error term above is
\[
\ll   \frac x{\log x} \sum_{\substack{ a,b\leq y \\ x/y^2<m\leq 2x\\ P(m)\leq y}}   \frac{\Lambda (a)}a  \frac{\Lambda (b) }b  \frac 1m \min\left \{ 1 , \frac 1{T|\log  (x/abm)|} \right\} .
\]
The contribution of the triples $abm$ with $abm\leq x/2$ or $>2x$ is 
\[
\ll   \frac x{T \log x} \sum_{\substack{ a,b\leq y \\ x/y^2<m\leq 2x\\ P(m)\leq y}}   \frac{\Lambda (a)}a  \frac{\Lambda (b) }b  \frac 1m 
\ll \frac {V\Psi(x,y)(\log y)^3}{T \log x },
\]
by \eqref{SmoothCompare}.
Next we look at those $abm\in (x(1-\frac {k+1}T), x(1-\frac {k}T)]$ for some $k,\ 1\leq k\leq T/2$, which contribute 
\[
\ll   \frac x{k \log x} \sum_{\substack{ a,b\leq y  }}   \frac{\Lambda (a)}a  \frac{\Lambda (b) }b  
\sum_{\substack{\frac x{ab} (1-\frac {k+1}T) <m\leq \frac x{ab} (1-\frac {k}T)\\ P(m)\leq y}} \frac 1m \ll  \frac {V\Psi(x,y)(\log y)^2}{kT \log x},
\]
 since by \eqref{SmoothShorts} we have\[
\sum_{\substack{\frac x{ab} (1-\frac {k+1}T) <m\leq \frac x{ab} (1-\frac {k}T)\\ P(m)\leq y}} \frac 1m \ll 
\frac {\Psi(x/ab,y)}{T\cdot x/2ab} \ll \frac VT \frac{\Psi(x,y)}x .
\]
The same argument works when $abm\in [x(1+\frac {k}T), x(1+\frac {k+1}T))$ for some $k,\ 1\leq k\leq T$. Finally if
$abm\in (x(1-\frac {1}T), x(1+\frac {1}T))$, we get the same bound as with $k=1$ above, and so our total truncation error is 
\[
\ll \frac {V\Psi(x,y)(\log y)^2\log T}{T \log x} \ll \frac {\Psi(x,y)  }{\varphi(q) (\log x)^{A-1}} ,
\]
by our choice of $T$.

\subsection{Back to the main term}  Let $T_q=T_q(A, x)$ be the set of characters $\chi \pmod q$ for which 
there exists  $x/y^2<X\leq x$,  for which
\begin{equation} \label{Hyp5}
 |S_f(X,\chi)|   >   { \Psi(X,y)}/{V^2(\log x)^{2A+2+\varepsilon}}  .
\end{equation}
Note that this is a weaker hypothesis than~\eqref{Hyp2} in that the range for $X$ is slightly shortened. Summing what remains of the main term over all $\chi \not\in    T_q$, and dividing through by $\varphi(q)$ we obtain the bound  
\[
\ll x \int_{\substack{s=c+it \\ |t|\leq T }}     \int_{\beta=0}^\infty  \frac 1{\varphi(q)} \sum_{ \chi  \not\in    T_q} 
 \left|  \sum_{a \leq y} \frac{ \Lambda_f(a)\overline\chi(a)  } { a^{s}} \right| \left|  \sum_{ b\leq y}   \frac{ \Lambda_f(b)\overline\chi(b) } { b^{\beta+s}} \right| \left| \sum_{\substack{x/y^2<m\leq 2x\\ P(m)\leq y}}   \frac{ f(m)\overline\chi(m) } { m^{\beta+s}}   \right| d\beta    \frac{dt }{1+|t|}.
\]
\subsection{A more subtle truncation:   $T\geq |t|>U$}  \label{subtle} In this range we extend the character sum
to all $\chi \pmod q$, and use the trivial upper bound, with $\beta'=\beta+\frac 1{\log x}$, to obtain
\begin{equation} \label{Trivial M}
\left|  \sum_{\substack{x/y^2<m\leq 2x\\ P(m)\leq y}}  \frac{ f(m)\overline\chi(m) } { m^{\beta+s}}   \right| \leq 
 \sum_{\substack{x/y^2<m\leq 2x\\ P(m)\leq y}}  \frac{ 1 } { m^{\beta+c}}  \ll \frac{V \Psi(x,y)  } {x \beta' (x/y^2)^{\beta'}} \ll 
 \frac{ V\Psi(x,y) \log x } { x (x/y^2)^{\beta }} .
\end{equation}
Therefore our integral is $V\Psi(x,y) \log x$ times
\[
\ll   \int_{\substack{s=c+it \\ T\geq |t|>U}}        \int_{\beta=0}^\infty  \frac 1{\varphi(q)} \sum_{ \chi \pmod q} 
 \left|  \sum_{a \leq y} \frac{ \Lambda_f(a)\overline\chi(a)  } { a^{s}} \right| \left|  \sum_{ b\leq y}   \frac{ \Lambda_f(b)\overline\chi(b) } { b^{\beta+s}} \right| \frac{d\beta} {  (x/y^2)^{\beta }}      \frac{dt }{1+|t|}.
\]
We now use the Cauchy-Schwarz inequality, so that the square of this integral is 
\[
\leq  \int_{\substack{s=c+it \\ T\geq |t|>U}}        \int_{\beta=0}^\infty  \frac 1{\varphi(q)} \sum_{ \chi \pmod q} 
 \left|  \sum_{a \leq y} \frac{ \Lambda_f(a)\overline\chi(a)  } { a^{s}} \right|^2 \frac{d\beta} {  (x/y^2)^{\beta }}    \frac{dt }{1+|t|}
\]
times
\[
  \int_{\substack{s=c+it \\ T\geq |t|>U}}        \int_{\beta=0}^\infty  \frac 1{\varphi(q)} \sum_{ \chi \pmod q} 
  \left|  \sum_{ b\leq y}   \frac{ \Lambda_f(b)\overline\chi(b) } { b^{\beta+s}} \right|^2 \frac{d\beta} {  (x/y^2)^{\beta }}     \frac{dt }{1+|t|} .
\]
We begin by noting that for $s=c+it$,
\[
\begin{split}
 \frac 1{\varphi(q)} \sum_{ \chi \pmod q} 
\left| \sum_{a\leq y}\frac{ \Lambda_f(a)\overline\chi(a) }{a^{s}}    \right|^2 &=
\sum_{a,b\leq y}\frac{ \Lambda_f(a)\overline{\Lambda_f(b)}}{(ab)^{c}(a/b)^{it}}  \frac 1{\varphi(q)} \sum_{ \chi \pmod q}  \overline\chi(a) \chi(b) \\
&=
\sum_{\substack{a,b\leq y \\ a\equiv b \pmod q}}\frac{ \Lambda_f(a)\overline{\Lambda_f(b)}}{(ab)^{c}} (a/b)^{-it} .
\end{split}
\]
Therefore, as $|\Lambda_f(a)\overline{\Lambda_f(b)}/(ab)^{c}|\leq \Lambda(a) \Lambda(b)/ab$, and as $f$ is only supported on primes $>qL$,
\[
 \int_{|t|=U}^T   \frac 1{\varphi(q)} \sum_{ \chi \pmod q} 
\left| \sum_{a\leq y}\frac{ \Lambda_f(a)\overline\chi(a) }{a^{s}}    \right|^2   \frac{dt }{1+|t|} \leq
\sum_{\substack{qL<a,b\leq y \\ a\equiv b \pmod q}}\frac{ \Lambda(a) \Lambda(b)}{ab} 
\left|   \int_{|t|=U}^T   (a/b)^{-it}   \frac{dt }{1+|t|} \right| .
\]
Now we obtain two different bounds on this integral: Trivially $|(a/b)^{it}|=1$ so the integral is $\ll \log (T/U)\leq \log T$.
Alternatively, working with $U<t\leq T$ (the integral for $-t$ being entirely analogous), for $a\ne b$
\[
  \int_{U}^T   (a/b)^{-it}   \frac{dt }{1+t}  = \frac{(a/b)^{-it}}{i\log(b/a) (1+t)} \bigg|_{U}^T + \int_{U}^T   \frac{(a/b)^{-it}}{i\log(b/a)  }  \frac{dt }{(1+t)^2} \ll \frac{1}{|\log(b/a)| U} .
 \]
 Therefore,   our original $a$-integral is
 \[
\ll  \sum_{\substack{qL<a\leq b\leq y \\ a\equiv b \pmod q}}\frac{ \Lambda(a) \Lambda(b)}{ ab } 
 \min \left\{ \log T,   \frac{1}{|\log(b/a)| U}  \right\}  \cdot 
  \int_{\beta=0}^\infty  \frac{d\beta} {  (x/y^2)^{\beta }}      .
\]
The $\beta$-integral is $\ll 1/\log x$. The sum over the terms with $b=a$ is
\[
\ll  \sum_{qL<a \leq y }\frac{ \Lambda(a)^2}{ a^2 } 
\log T  \ll    \frac{\log q+\log\log x}{qL} \log T \ll
\frac 1{V\varphi(q)(\log x)^{A-1}}.
\]
For those $b$ in the interval $a<b<2a$ we may write $b=a+nq$ with $1\leq n\leq a/q$.
Forgetting these are meant to be prime we get an upper bound
\begin{align*}
&\ll  \sum_{qL<a\leq y}\frac{ \Lambda(a) \log a}{ a^2 } 
\sum_{n=1}^{[a/q]} \min \left\{ \log T,   \frac{a}{nqU}  \right\}  \\
&\ll \sum_{qL<a\leq y}\frac{ \Lambda(a) \log a}{ a^2 } \cdot \frac a{qU} \log(U\log T)\\
&\ll \frac{ \log(U\log T)}{qU} (\log y)^2\ll \frac 1{V\varphi(q)(\log x)^{A-1}} .
\end{align*}
Next we partition the remaining $b$-values according to
$2^ka<b\leq 2^{k+1}a$ for $1\leq k\leq {K-1}$ where $2^K\asymp y/a$, so our sum becomes
\begin{align*}
&\ll \sum_{qL<a\leq y} \frac{ \Lambda(a)  }{ a } \sum_{k=1}^{K-1}   \frac{1}{k U}   \sum_{\substack{2^ka<b\leq 2^{k+1}a \\ b\equiv a \pmod q}}\frac{  \Lambda(b)}{ b } \\
&\ll \sum_{qL<a\leq y} \frac{ \Lambda(a)  }{ a } \sum_{k=1}^{K-1}   \frac{1}{k   U\phi(q)}   
\ll \frac{\log y}{    U\phi(q)} \log K\ll \frac 1{V\varphi(q)(\log x)^{A-1}} 
\end{align*}
using the Brun-Titchmarsh Theorem. 

Collecting all this together we see that  the $a$-integral is $\ll 1/(V\varphi(q)(\log x)^{A})$. The same argument works for the $b$-integral (where we use that $c+\beta>1$ rather than $c>1$), and so the total contribution of this part of the $t$-integral is 
\[
\ll \frac{ \Psi(x,y)  }{ \varphi(q)(\log x)^{A-1} } ,
\]
which is acceptable.

\subsection{Consequences of~\eqref{Hyp5}}
If $\chi\not\in  T_q$ then, by \eqref{Hyp5},  we have
\[
\begin{split}
\sum_{ x/y^2<m\leq X}     f(m)m^{-it} \overline\chi(m) &= \int_{x/y^2}^X \frac{dS_f(h,\chi)  }{h^{it}} = 
\frac{ S_f(h,\chi) }{h^{it}}\bigg|_{x/y^2}^X + it \int_{x/y^2}^X \frac{ S_f(h,\chi) }{h^{1+it}}dh \\
& \ll  
\frac { \Psi(X,y)}{V^2(\log x)^{2A+2+\varepsilon}}   +\frac { |t|}{V^2(\log x)^{2A+2+\varepsilon}}  \int_{x/y^2}^X \frac{\Psi(h,y)}h dh  \\
& \ll \frac { (1+|t|)  \Psi(X,y)}{V^2(\log x)^{2A+2+\varepsilon}}  ,
\end{split}
\]
as $\Psi(h,y)/h\ll (\Psi(X,y)/X)(X/h)^{1-\alpha}$, and $\int^X_{x/y^2} (X/h)^{1-\alpha} dh \ll X$.

Calling this new sum $S(X)$ then, for $\delta=c-1+\beta\geq 1/\log x$ where $ \beta\geq 0$, we have, using \eqref{SmoothCompare},
\begin{equation} \label{Hyp}
\begin{split}
 \sum_{x/y^2<m\leq 2x}   \frac{ f(m)\overline\chi(m) } { m^{c+\beta+it}}  &=
 \int_{x/y^2}^{2x}   \frac{dS(X)}{X^{1+\delta}} 
 \ll  \frac {(1+|t|)V \Psi(x,y)/x }{V^2(\log x)^{2A+2+\varepsilon}}     \frac{1}{\delta (x/y^2)^{\delta}}\\
&  \ll  \frac {(1+|t|) \Psi(x,y)/x}{V(\log x)^{2A+1+\varepsilon}  x^{\eta\beta}},
  \end{split}
\end{equation}
for $1\leq y\leq x^{(1-\eta)/2}$.

\subsection{Bounds under~\eqref{Hyp5} for  $ |t|\leq U$}   If $\chi\not\in T_q$ then, by \eqref{Hyp}, our remaining main term is 
\[
\ll  \frac {\Psi(x,y)}{V(\log x)^{2A+1+\varepsilon} } \int_{\substack{s=c+it \\ |t|\leq U}}  \left(   \int_{\beta=0}^\infty x^{-\eta\beta}    \left|  \sum_{a \leq y} \frac{ \Lambda_f(a)\overline\chi(a)  } { a^{s}} \right| \left|  \sum_{ b\leq y}   \frac{ \Lambda_f(b)\overline\chi(b) } { b^{\beta+s}} \right|   d\beta\right)   dt.
\]
We now sum over all such $\chi\not\in T_q$ and divide by $\varphi(q)$, so that the inner integral becomes
\[
 \int_{\beta=0}^\infty x^{-\eta\beta}    \frac 1{\varphi(q)} \sum_{ \chi \pmod q}   \left|  \sum_{a \leq y} \frac{ \Lambda_f(a)\overline\chi(a)  } { a^{s}} \right| \left|  \sum_{ b\leq y}   \frac{ \Lambda_f(b)\overline\chi(b) } { b^{\beta+s}} \right|   d\beta.
\]
We follow the plan of section \ref{subtle}, Cauchying and then computing the separated integrals, but here we do not bring the integral over $t$ inside, simply using the bound $|(a/b)^{it}|=1$. 
Therefore the square of this inner integral is, by Cauchy-Schwarz,
\[
\leq  \int_{\beta=0}^\infty x^{-\eta\beta}  \frac 1{\varphi(q)} \sum_{ \chi \pmod q} 
\left| \sum_{a\leq y}\frac{ \Lambda_f(a)\overline\chi(a) }{a^{s}}    \right|^2  d\beta 
\]
times the analogous integral for $b$. 
Expanding the square in the first integral we have 
\[
 \frac 1{\varphi(q)} \sum_{ \chi \pmod q} 
\left| \sum_{a\leq y}\frac{ \Lambda_f(a)\overline\chi(a) }{a^{s}}    \right|^2  
\leq  \sum_{\substack{q^2<a,b\leq y \\ a\equiv b \pmod q}}\frac{ \Lambda (a) {\Lambda (b)}}{ab}   
\]
as $qL\geq q^2$,
and then the integral over $\beta$ is $\ll 1/\log x$. We get to the same point with the $b$-integral.
 By the Brun-Titchmarsh theorem we have
\[
\sum_{\substack{q^2<a,b\leq y \\ a\equiv b \pmod q}}\frac{ \Lambda (a) {\Lambda (b)}}{ab} =
\sum_{\substack{q^2<a \leq y }}\frac{ \Lambda (a)  }{a}
\sum_{\substack{q^2<b\leq y \\ b\equiv a \pmod q}}\frac{  {\Lambda (b)}}{b} \ll \frac{(\log y)^2}{\varphi(q) }
 \ll \frac{(\log x)^2}{\varphi(q) }.
\]
Our remaining main term, summed over all $\chi \not\in   T_q$, divided through by $\varphi(q)$, is therefore
\[
\ll  \frac {  \Psi(x,y)}{V(\log x)^{{2A+1+\varepsilon}} } \int_{\substack{s=c+it \\ |t|\leq U}}   \frac{\log x}{\varphi(q) } dt \ll  \frac  {  \Psi(x,y)}{\varphi(q)(\log x)^{{A-1}} }.
\]

\subsection{Combining results and partial summation}
Combining the above we have therefore proved  that
\[
 \frac 1{\varphi(q)}  \sum_{ \substack{\chi \pmod q \\ \chi \not\in   T_q(A,x) }}  \left|\sum_{x<n\leq 2x} f(n)\overline\chi(n) \log n\right|
\ll   \frac{ {  \Psi(x,y)}}{\varphi(q) (\log x)^{A-1}} ,
\]
provided $ x^{(1-\eta)/2}\geq y\geq V^3q^2(\log x)^{2A-2}$. Select integer $k$ so that $2^{k\alpha}\asymp q(\log x)^{A}$. Since
\[
\bigcup_{j=1}^k \ T_q(A,x/2^j)  \subset \mathcal T_{q,f}(A,2^ky^2),
\]
we deduce that 
\[
 \frac 1{\varphi(q)}  \sum_{ \substack{\chi \pmod q \\ \chi \not\in   \mathcal T_q(A,2^{k}y^2) }}  \left|\sum_{X/2<n\leq X} f(n)\overline\chi(n) \log n\right|
\ll   \frac{ { \Psi(X,y)}}{\varphi(q) (\log x)^{A-1}} ,
\]
for $X=x,x/2,x/4,\ldots, x/2^{k-1}$. Summing these up, as
\[ 
\sum_{j=0}^{k-1} \Psi(x/2^j,y) \ll \Psi(x ,y) \sum_{j=0}^{k-1} \frac 1{2^{j\alpha}} \ll \Psi(x ,y),
\]
and as $\Psi(x/2^k,y)\ll \Psi(x,y)/2^{k\alpha} \ll \Psi(x,y)/q(\log x)^{A} $, we deduce that
\[
 \frac 1{\varphi(q)}  \sum_{ \substack{\chi \pmod q \\ \chi \not\in   \mathcal T_q(A,2^{k}y^2)  }}  \left|\sum_{ n\leq x} f(n)\overline\chi(n) \log n\right|
\ll   \frac{   \Psi(x,y) }{\varphi(q) (\log x)^{A-1}} ,
\]
provided $y\geq V^3q^2(\log x)^{2A-2}$ and $y^2(q(\log x)^A)^{1/\alpha}\leq x^{1-\eta}$.

We deduce that 
\[
 \frac 1{\varphi(q)}  \sum_{ \substack{\chi \pmod q \\ \chi \not\in   \mathcal T_q(A,2^{2k}y^2)  }}  \left|\sum_{ n\leq X} f(n)\overline\chi(n) \log n\right|
\ll   \frac{\Psi(X,y) }{\varphi(q) (\log x)^{A-1}} ,
\]
for all  $X$ in the range $x/(q(\log x)^A)^{1/\alpha}\leq X\leq x$, provided 
\[ y\geq V^3q^2(\log x)^{2A-2} \ \text{ and } y(q(\log x)^A)^{1/\alpha}\leq x^{(1-\eta)/2}. \] 
Then \eqref{smallsupport} holds by partial summation, as 
$\int_2^x (\Psi(t,y) /t)dt \ll \Psi(x,y)$.

Now as $y\geq (\log x)^{1+\varepsilon}$ we have $V\leq (\log x)^2$, so the first range hypothesis holds
when $q\leq y^{1/2} /(\log x)^{ A+2}$. If $y\leq x^{1/2-\eta}$ then the second hypothesis holds when
$q \leq x^{ \eta\alpha/2}/(\log x)^A$. Now this follows from the first range hypothesis if $y\leq x^\eta$; if $y$ is larger then this simplifies to
$q \ll x^{ \eta/2}/(\log x)^A$. 
 \end{proof}
 
 \subsection{Proof of Corollary \ref{ErrorfapsA=0}} We take $T=\varphi(q)$ in the same argument, and  extend the range   in section \ref{subtle} to $  |t|\leq T$, getting rid of the need for the final range. In the calculations in that section we replace $U$ by $U+1$ in our bounds, and end up with a contribution $\ll V\Psi(x,y)(\log x)^{2+\varepsilon}/\varphi(q) $. With our new choice of $T$, the contribution with $|t|>T$ also is bounded by the same quantity, so, going through the same process, we end up with the upper bound $\ll V\Psi(x,y)(\log x)^{1+\varepsilon}/\varphi(q)\ll  \Psi(x,y)(\log x)^{3+\varepsilon}/\varphi(q) $ on our sum, as desired.

 \subsection{Extension to $f$ supported on any primes $\leq y$}
 Given $f$ supported on primes $p \leq y$, define the multiplicative function $g$ with $g(p^k)=0$ if $p\leq qL$, and 
 $g(p^k)=f(p^k)$ if $p> qL$.  Therefore if  $\psi \pmod r$ is a primitive character that  induces $\chi \pmod q$ then 
 $g(m) \overline{\chi}(m) = f(m) \overline{\psi}(m) 1_{(m,P)=1}$ where $P=\prod_{p\leq qL} p$.
 Let $h(.)$ be the multiplicative function which is supported only on the prime powers $p^k$, for which $p\leq qL$ but does not divide $r$, and defined so that $(h*f\overline\psi)(p^k)=0$ if $k\geq 1$. Thus $h*f\overline\psi = g\overline\chi$  and so 
\begin{equation} \label{Ident}
S_{g}(x,\chi) = \sum_{m\leq x} h(m) S_{f}(x/m,\psi) .
\end{equation}
We note that $h\psi\in \mathcal C$ as $f\in \mathcal C$, so that each $|h(m)|\leq 1$.

 \begin{proof}[Deduction of Proposition \ref{Errorfaps1} from Proposition \ref{Errorfaps1q}]
Assume that   $\psi$  is a primitive character mod $r$ with $r\leq R$,  such that
\[ |S_{f}(X ,\psi)| \leq \Psi(X,y)/V^2(\log x)^C\] for all $X$ in the range $x/z_0< X\leq x$ where $z_0= z_1 (qL)^{\log\log x}$ and $z_1 = \exp((\log\log x)^2)$.
Now let $x/z_1\leq X\leq x$. Then $|S_{f}(X/m,\psi)| \leq \Psi(X/m,y)/V^2(\log x)^C$ for $m\leq (qL)^{\log\log x}$
and so, by \eqref{Ident} and \eqref{alphaCompare}, we obtain 
\[
\begin{split}
|S_{g}(X,\chi)| & \ll  \ \frac {\Psi(X,y)}{V^2(\log x)^C} \sum_{\substack{m\leq (qL)^{\log\log x} \\ P(m)\leq qL}}   \frac 1{m^\alpha}+ \Psi(X,y)\sum_{\substack{(qL)^{\log\log x}<m\leq X\\ P(m)\leq qL}}  \frac 1{m^\alpha} \\
 & \ll  \ \frac {\Psi(X,y)}{V^2(\log x)^C}  \exp \left( \sum_{p\leq qL}   \frac 1{p^\alpha}\right) + \Psi(X,y)\sum_{\substack{(qL)^{\log\log x}<m\leq X\\ P(m)\leq qL}}  \frac 1{m^\alpha} .
   \end{split}
\]  
Now if $q\leq R$ then
\begin{equation} \label{Restrictq}
\log q< \log L \leq  \frac{\log y \log\log\log x}{3\log u}
\end{equation}
(where the factor $1/3$ can be replaced by $1/2-\ee$). Therefore $(qL)^{1-\alpha} \leq (\log\log x)^{2/3+o(1)}$ as $y^{1-\alpha}=u^{1+o(1)}$, and so
\[
\sum_{p\leq qL}   \frac 1{p^\alpha} \ll \log 1/\alpha + \frac{(qL)^{1-\alpha} }{(1-\alpha)\log qL} \leq (\log\log x)^{2/3+o(1)}.
\]
Also if $M=(qL)^v$ with $v\geq  \log\log x$ then
\[
\sum_{\substack{m\sim M \\ P(m)\leq qL}}  \frac 1{m^\alpha} \ll M^{1-\alpha} v^{-v+o(v)} \ll \left( (qL)^{1-\alpha}/v^{1+o(1)}\right)^v \ll (\log\log x)^{-\frac 13 (v+o(v))} .
\]
Combining this information yields  that 
\[
|S_{g}(X,\chi)| \ll \frac {\Psi(X,y)}{V^2(\log x)^{C+o(1)}\ }
\]
 Therefore  \eqref{Hyp2} fails  taking $C= 2A+2+2\varepsilon$, that is  $\chi\not\in \mathcal T_{q,g}(A,z_1)$.
Proposition \ref{Errorfaps1q} then implies that, for $z_1=y^2(q^2(\log x)^{2A})^{1/\alpha}$,
\begin{equation} \label{smallsupport2}  
    \sum_{ \substack{\chi \pmod q \\ \chi \not\in  \mathcal T_q^*(A)}} \left|  S_g(x,\chi) \right| \ll    \frac{\Psi(x,y)} {(\log x)^{A}} .
\end{equation}
  
 Write each $n\leq x$ as $n=n'N$ with $P(n')\leq qL$, and
$p|N \implies p> qL$, so that $f(n)=f(n')g(N)$, and therefore (after renaming $n'$ by $n$)
\[
S_f(x,\chi) = \sum_{\substack{n\geq 1\\ P(n)\leq qL}} f(n) S_g(x/n,\chi) .
\]

Now assume that  $\psi\not\in \Xi(C)$, so that 
 $|S_{f}(X ,\psi)| \leq \Psi(X,y)/V^2(\log x)^C$ for all   $x/z_2< X\leq x$ where $z_2= z_0 (qL)^{\log\log x}$. Assuming that $x/z_2 \geq x^{\eta}$, \eqref{smallsupport2} holds with $x$ replaced by $X$, for every $X$ in the range $x/(qL)^{\log\log x}\leq X\leq x$.
Then  we have, by \eqref{smallsupport2}, Corollary \ref{ErrorfapsA=0} and \eqref{alphaCompare},
\[
 \begin{split}
   \sum_{ \substack{\chi \pmod q \\ \chi \not\in   \Xi_q(C)}} \left|  S_f(x,\chi) \right| &\leq
   \sum_{\substack{n\leq (qL)^{\log\log x}\\ P(n)\leq qL}}  \sum_{ \substack{\chi \pmod q \\ \chi \not\in   \Xi_q(C) }}     |S_g(x/n,\chi) |
   +    \sum_{\substack{n> (qL)^{\log\log x}\\ P(n)\leq qL}}  \sum_{ \substack{\chi \pmod q   }}     |S_g(x/n,\chi) | \\
   &\ll    \frac{\Psi(x,y)} {(\log x)^{A}}  \sum_{\substack{n\leq (qL)^{\log\log x}\\ P(n)\leq qL}}   \frac 1{n^\alpha}   + \Psi(x,y)(\log x)^{3+\varepsilon}   \sum_{\substack{n> (qL)^{\log\log x}\\ P(n)\leq qL}}      \frac 1{n^\alpha} \\
    &\ll    \frac{\Psi(x,y)} {(\log x)^{A+o(1)}}   ,
         \end{split}
\]  
proceeding for these two sums just as we did above.

The claimed estimate then follows by taking $A=B+\varepsilon$ and so $C= 2B+2+4\varepsilon$. To guarantee that $x/z_2\geq x^{\eta}$ we need that 
\[
y(q(\log x)^A)^{1/\alpha} (qL)^{\log\log x} \leq x^{(1-\eta)/2},
\]
and we already have the restrictions $y\geq V^3q^2(\log x)^{2A-2}$ and \eqref{Restrictq}. 
These all follow from the bounds on $q$ and $y$ given in the hypothesis. All that remains now is to bound $|\Xi|$.
\end{proof}

\subsection{Bounding the size of $\Xi(C)$} \label{LargeS}
The large sieve gives that 
\[
\sum_{r\leq \sqrt{X}} \sum_{\psi \text{ primitive mod } r} \left|  \sum_{  m\leq X}     f(m) \overline\psi(m) \right|^2 \ll X \Psi(X,y) .
\]
Therefore, for a given $X$ in the range  $ x^{\eta}\leq X\leq x$,  the number of exceptional $\psi$ with $r\leq R$ is $\ll (X/\Psi(X,y)) (\log x)^{2C} u^{O(1)}\ll u^{u+o(u)} (\log x)^{2C}$. 

If \eqref{Hyp3} holds for $X$, and $|X'-X|\ll X/V^4(\log x)^{C}$ then \eqref{Hyp3} holds for $X'$. So to obtain the full range for $X$ we sample at $X$-values spaced by gaps of length $\gg X/V^4(\log x)^{C}$. Therefore the 
  number of exceptional $\psi$ with $r\leq R$   in our range  is $\ll u^{u+o(u)} (\log x)^{3C+1}$,   as claimed.

\section{Deduction of several Theorems}\label{sec:deduce-many-thms}
\subsection{Proof of  two Theorems on $f$ in arithmetic progressions}  
 
\begin{proof} [Proof of Theorem \ref{Cor:Result2} ]   Let $J=k+1$ so that $\frac 12 \leq B=1-\varepsilon<1-1/\sqrt{J}$ in Corollary~\ref{MathResult2Cor**}. Let $w=(\log x)^{2B}$ and 
 $ \exp(C(\log\log x)^2)<Q\leq x^{1/2-\delta}$, with $C>4B^2$. Let $\Xi=\{ \psi_1,\ldots, \psi_{J-1}\}$ so that 
 $\Delta_k=\Delta_{\Xi}$ and \eqref{CharHyp} holds. Then, by Corollary \ref{MathResult2Cor} and Corollary \ref{MathResult2Cor**} we deduce that 
\[ 
 \sum_{q \sim Q}   \max_{(a,q )=1}  \left| \Delta_{\Xi}(f,x;q,a) \right| 
  \ll    \frac x {(\log x)^{B}} .
\]
\end{proof}

\begin{proof} [Deduction of Corollary \ref{Thm: SW}] 
We apply Proposition \ref{Using SW} whose hypotheses are satisfied using  Theorem   \ref{Cor:Result2} for any $C>0$, suitably adjusting the value of $\varepsilon$.   
\end{proof}

\subsection{On average, supported only on the smooths }

\begin{proof} [Proof of Theorem \ref{Keep Xi}]
We proceed as in the proof of Theorem~\ref{MathResult2}, but now with the parameters $Y=\exp((\log\log x)^{2}), Z=y$ and  $w=(\log x)^{2A}$. Since $Q\leq x^{1/2}/y^{1/2}(\log x)^A$, Lemma \ref{BilinBound} implies that
\[  E_{\text{bilinear}}    \ll x \left(\frac{1}{(w\log w)^{1/2}} + \frac{Y^{0.1} + (\log x)^{1/2}}{\pi(Y)^{1/2}}\right) +QZ^{1/2}x^{1/2} \ll \frac x{(\log x)^A}. \]
Moreover, we have
\[ \frac{T}{Y \log Y} \ll \frac x{(\log x)^A}.\]
The key difference here is in the $E_{\text{sieve}}$ term. Since $f$ is only supported on the $y$-smooth integers, we now have
\[ E_{\text{sieve}} \leq \sum_{n \leq x} | F(n)| \1_{(n,\prod_{p>Y}p) = 1} \leq \sum_{q\sim Q} \left(\Psi(x,Y,a_q,q) + \frac 1{q_r}  \Psi(x,Y,a_q,q_s)\right).
\]
By \eqref{BVsmooth}, followed by  \eqref{SmoothUB}, this quantity is 
\[
\ll \sum_{q\sim Q} \frac 1{q} \Psi(x,Y)  + \frac x{(\log x)^A}  \ll  \frac x{(\log x)^A}.
\]
Thus from Proposition~\ref{prop:ramare} we deduce the following variant on Theorem \ref{MathResult2}:
\[
  \sum_{q \sim Q} \max_{(a,q)=1} \left|  \sum_{\substack{n \leq x\\ n \equiv a\pmod{q}}} f(n) -\frac{1}{ q_r} 
\sum_{\substack{n \leq x\\ n \equiv a\pmod{q_s}}} f(n)  \right| \ll   \frac x{(\log x)^A}.
 \]
 
We will now apply Proposition \ref{MainCor}  to the sums of $f(n)$ over $n$ in the $w$-smooth arithmetic progressions $a \pmod{q_s}$. Therefore the only relevant $\chi\in \Xi$ will be those with $w$-smooth conductors. We select $B$ so that $6B+7<2A$, which implies that 
 \eqref{CharHyp} holds, provided $x^{1/2-\ee}\geq Q\geq   x^{  \varepsilon/(2\log\log x)}$.
 By Corollary \ref{MathResult2Cor} suitably amended with this input, we then deduce that
\begin{equation}\label{eq:boum0}
 \sum_{q \sim Q} \max_{(a,q)=1} \left|  \Delta_{\Xi}(f,x;q,a)  \right|  \leq 
  \frac { 1}{\log w}   
  \sum_{\substack{ q_s\leq 2Q \\ P(q_s)\leq w}}    \max_{(a,q_s)=1}     | \Delta_{\Xi}(f,x;q_s,a) | + O\left(  \frac x{(\log x)^A}\right) .
 \end{equation}

By Proposition \ref{MainCor} we have, for $R= x^{  \varepsilon/(3\log\log x)}$
\begin{equation}\label{eq:boum7}
 \frac { 1}{\log w}   
  \sum_{\substack{ q_s\leq R \\ P(q_s)\leq w}}    \max_{(a,q_s)=1}     | \Delta_{\Xi}(f,x;q_s,a) | \ll
   \frac { 1}{\log w}   
  \sum_{\substack{ q_s\leq R \\ P(q_s)\leq w}}   \frac 1{\varphi(q_s)}  \frac{x} {(\log x)^{B}} \ll \frac{x} {(\log x)^{B}} .
 \end{equation}
For the remaining $q_s$ we use the trivial upper bound $ | \Delta_{\Xi}(f,x;q ,a) |\ll (1+|\Xi_{q}|)x/q$, to obtain
\[
\begin{split}
& \frac { 1}{\log w}   
  \sum_{\substack{ R\leq q_s\leq 2Q \\ P(q_s)\leq w}}    \max_{(a,q_s)=1}     | \Delta_{\Xi}(f,x;q_s,a) | \ll
  \frac { x}{\log w}   
  \sum_{\substack{ R\leq q\leq 2Q \\ P(q)\leq w}} \frac {1+|\Xi_{q}|}{q}  \\
  & \leq
     \frac { x}{\log w}   \left(  \left( 1+  \sum_{\substack{\chi\in \Xi \\ r_\chi\leq Y} }  \frac 1{r_\chi} \right) \sum_{\substack{ q\geq R/Y \\ P(q)\leq w}} \frac 1{q} + \sum_{\substack{\chi\in \Xi \\ r_\chi>Y} }  \frac 1{r_\chi} \
      \sum_{\substack{ q\geq 1  \\ P(q)\leq w}} \frac {1}{q} \right) \ll \frac{x} {(\log x)^{A}},
      \end{split}
 \]
 using  \eqref{CharHyp}, and that $ \sum_{\chi\in \Xi,\ r_\chi>Y } 1/{r_\chi} \leq |\Xi|/Y$, and   our estimates on smooth numbers (as $R/Y>w^{\log\log x}$). We therefore deduce \eqref{eq:boum5}.
 \end{proof}

 
  \begin{proof} [Deduction of Corollary \ref{MathResult3}]  We apply Proposition \ref{Using SW} whose hypotheses are satisfied using  Theorem   \ref{Keep Xi} for any $C>6B+7$.   
\end{proof}

\section{Breaking the $x^{1/2}$-barrier}

To break the $x^{1/2}$-barrier, we need to reduce the $Q^2$ in the upper bound in Lemma~\ref{BilinBound}. This term arises in the estimates that the number of terms in various arithmetic progressions is the length of that progression plus $O(1)$.  Following ideas from~\cite{BFI1,FI,Zh} and others, we will be more precise about all those ``$O(1)$''s by using Fourier analysis  to obtain some cancellation

We will be able to do this when   the residue classes $a_q$ do not vary with $q$:

 Let $1\leq |a|\ll  Q\leq x^{\frac{20}{39}-\delta}$ for some small fixed $\delta>0$, and consider the residue classes $a\pmod q$ when $q \sim Q$ and $(a,q) = 1$. Let $F$ be the function defined by
\[ F(n) = \sum_{\substack{q \sim Q \\ (a,q)=1}} \xi_q \left(\1_{n \equiv a\pmod{q}} - \frac{1}{q_r} \1_{n \equiv a\pmod{q_s}}\right) \]
when $n \neq a$, and set $F(a) = 0$. Writing $\tau$ for the divisor function,  we have $|F(n)| \ll \tau(n-a) \leq x^{o(1)}$ for each $n \neq a$ and $n \leq x$, and so $ \|F\|_{\infty}\leq x^{o(1)}$. Moreover we have the $L^2$-bound:
\begin{equation*}  
\sum_{n \leq x} |F(n)|^2 \ll \sum_{n \leq 2x} \tau(n)^2 \ll x(\log x)^3. \end{equation*}

\begin{proposition}[Bilinear term beyond the $x^{1/2}$-barrier]\label{prop:bilinear-beyond1/2}
Let $F$ be defined as above, and fix $\delta>0$.
If $1\leq |a|\ll  Q\leq x^{\frac{20}{39}-\delta}$   and $P\leq x^\varepsilon$ for some very small $\varepsilon>0$ (depending on $\delta$) then
\[ \E_{p,p'\sim  P} \left| \sum_{m \leq \min(x/p, x/p')} F(pm) \overline{F(p'm)} \right|  \ll \frac{x}{P}\left(\frac{1}{w\log w} + \frac{P^{0.1}+(\log x)^3}{\pi(P)}\right) .   \]
\end{proposition}

We begin by noting that the contribution from those terms with $p=p'$ is bounded by
\[ 
\frac{1}{\pi(P)^2} \sum_{p\sim P} \sum_{m \leq x/p}|F(pm)|^2 \ll 
\frac{1}{\pi(P)^2} \sum_{p\sim P} \sum_{\substack{ n \leq 2x\\ n\equiv -a \pmod p}} \tau(n)^2
\]
as $|F(pm)|\ll \tau{pm-a}$. Since $\tau(n)^2$ is a positive valued multiplicative function, this sum is easily bounded by Shiu's Theorem, \cite{Shiu}, to be
\[
\ll \frac{x}{P} \cdot \frac{(\log x)^3}{\pi(P)}, 
\]
which is acceptable. Now fix $p \neq p'$ and analyze the inner sum over $m$. We insert a smooth weight into the sum over $m$: Let $\eta > 0$ be a small parameter to be chosen later. Let $\psi \colon \R \to [0,1]$ be a smooth approximation to the indicator function $\1_{[0,1]}$ with the following properties:
\begin{enumerate}
\item $\int\psi = 1$;
\item $\psi$ is supported on $[-\eta, 1+\eta]$, and $\psi=1$ on $[\eta,1-\eta]$;
\item $\|\psi^{(A)}\|_1 \ll_A \eta^{-A+1}$ for all $A \geq 1$. 
\end{enumerate}
By partial summation it follows that the Fourier transform $\widehat{\psi}$ decays rapidly:
\[ |\widehat{\psi}(h)| \ll_A (1+|h|)^{-A} \eta^{-A+1} \]
for all $A \geq 1$. We may replace the sum over $m$ by the smoothed version
 \[ \Sigma(p,p') := \sum_{m \in \Z} \psi(m/M) F(pm) \overline{F(p'm)}, \]
where $M = \min(x/p, x/p')$, at the cost of an error
\[ \ll \eta M \|F\|_{\infty}^2 \ll \eta x^{o(1)} \cdot \frac{x}{P}. \]
Let $0<\sigma\leq \min\{ \delta/2,1/75\}$ and assume $\varepsilon=\sigma/3$.
If we choose $\eta = x^{-\sigma}$,  then the error is 
acceptable since $P \leq x^{\varepsilon}$.   By the definition of $F$, we have
\[ \Sigma(p,p') = \sum_{\substack{q,q'\sim Q \\ (a,q) = (a,q') = 1}} \xi_q \xi_{q'} \sum_{m \in \Z} \psi(m/M) K(p,p',q,q';m), \]
where $K(p,p',q,q';m)$ is the expression
\[\left(\1_{pm\equiv a\pmod{q}} - \frac{1}{q_r} \1_{pm\equiv a\pmod{q_s}}\right) \left(\1_{p'm\equiv a\pmod{q'}} - \frac{1}{q_r'} \1_{p'm\equiv a\pmod{q_s'}}\right).  \]
The inner sum over $m$ can be written as a sum of four terms, the first of which is
\[ \sum_{m \in \Z} \psi(m/M) \cdot \1_{pm\equiv a\pmod{q}} \cdot \1_{p'm\equiv a\pmod{q'}}. \]
The sum is empty unless
\[ p'a \equiv pa \pmod{(q,q')} \Leftrightarrow p' \equiv p\pmod{(q,q')}. \]
When this holds, the sum should have a main term of $M/[q,q']$ and an error term of
\[ g(q,q') = \sum_{m \in \Z} \psi(m/M) \cdot \1_{pm\equiv a\pmod{q}} \cdot \1_{p'm\equiv a\pmod{q'}} -
\frac{M}{[q,q']}. \]
After summing over $p,p'$, the contributions from the four main terms lead to exactly $\Sigma_1 + \Sigma_2$, which were treated in the proof of Lemma~\ref{BilinBound}:
\[ \Sigma_1 + \Sigma_2 \ll \frac{x}{P}\left(\frac{1}{w\log w} + \frac{P^{0.1}+\log x}{\pi(P)}\right). \]
It suffices to control the error terms by showing that
\[ E_1 := \sum_{\substack{q,q' \sim Q \\ (a,q) = (a,q') = 1}} \xi_q\xi_{q'}  g(q,q') \ll x/P^2,  \]
\[ E_2 := \sum_{\substack{q,q' \sim Q \\ (a,q) = (a,q') = 1}} \xi_q\xi_{q'} q_r^{-1} g(q_s,q') \ll x/P^2,  \]
\[ E_2' := \sum_{\substack{q,q' \sim Q \\ (a,q) = (a,q') = 1}} \xi_q\xi_{q'} q_r'^{-1}  g(q,q_s') \ll x/P^2,  \]
\[ E_3 := \sum_{\substack{q,q' \sim Q \\ (a,q) = (a,q') = 1}} \xi_q\xi_{q'} (q_rq_r')^{-1}  g(q_s,q_s') \ll x/P^2,  \]
for any fixed $p \neq p'$. By symmetry, it suffices to prove the bounds for $E_1, E_2, E_3$. We start by analyzing $g(q,q')$ for fixed $q,q'$.

\begin{lemma}\label{lem:g(q,q')} Suppose that $1\leq |a|\ll  Q\leq x^{\frac{20}{39}}$   and $P\leq x^\varepsilon$.
Fix $q,q' \leq 2Q$ and $p,p'\sim P$ with $d = (q,q')$ satisfying $p \equiv p'\pmod{d}$ and $(a,qq')=1$. Write $q=d\ell$ and $q'=d\ell'$ so that $(\ell,\ell')=1$. Then
\[ g(q,q') = \frac{M}{q'} \sum_{0 < |h| \leq H}  \e_{p'\ell}(kh \cdot \overline{p\ell'}) \int_{|u| \leq 2d/Q} \psi(\ell u) \ \e_{d\ell'}(-Mhu) \dd u + O_{\sigma}(x^{-1/10   }),  \]
where $H = x^{2\sigma}Q^2/(dM)$ and $k = (p'-p)a/d$, assuming $0<\sigma \leq 1/75$ and $\varepsilon = \sigma/3$.
\end{lemma}

Here, and in the sequel, the notation $\overline{p\ell'}$ denotes the multiplicative inverse of $p\ell'$ modulo $p'\ell$, when it appears inside $\e_{p'\ell}(\cdot)$. The definition of $g(q,q')$ involves $[q,q']=d\ell\ell'$. A key advantage of Lemma \ref{lem:g(q,q')} is that the two variables $\ell$ and $\ell'$ are separated apart from a term of the form $\e_\ell(*\overline{\ell'})$.

\begin{proof}
Let $r$ be the unique solution modulo $[q,q']$ to the simultaneous congruence conditions
\[ p r \equiv a\pmod{q}, \ \ p' r \equiv a\pmod{q'}. \]
Then
\[ g(q,q') = \sum_{m \in \Z} \psi\left(\frac{m}{M}\right) \1_{m \equiv r\pmod{[q,q']}} - \frac{M}{[q,q']}. \]
By the definition of the Fourier transform $\widehat{\psi}$  and a change of variables we have
\[ 
 \widehat{\psi}\left(\frac{2\pi Mh}{[q,q']}\right)  = \widehat{\psi}\left(\frac{2 \pi M h}{d\ell\ell'}\right) = \ell \int_{|u| \leq 2/\ell} \psi(\ell u) \e\left(-\frac{Mhu}{d\ell'}\right) \dd u,
\]
and by  Poisson summation, we obtain
\[ g(q,q') = \frac{M}{[q,q']} \sum_{h \in \Z\setminus\{0\}} \widehat{\psi}\left(\frac{2\pi Mh}{[q,q']}\right) \e\left(\frac{rh}{[q,q']}\right). \]
Using the rapid decay of $\widehat{\psi}$,  the contribution to $g(q,q')$ from those terms with $|h| \geq H$ is
\[
\ll_A 1\big/\left(\frac{ \eta MH}{[q,q']}\right)^{A-1} \ll_{\sigma} x^{-1}
\]
by the choices of $\eta = x^{-\sigma}$ and $H$ and letting $A=\sigma^{-1}+1$.  We will prove that  
\begin{equation}\label{AlmostId}
\e\left(\frac{rh}{[q,q']}\right) = \e_{d\ell\ell'}(rh) =  \e_{p'\ell}(kh \cdot \overline{p\ell'})  + O\left( \frac {|h|}{Q}\right).
\end{equation}
The total contribution of these error terms, over all $h$ with $1\leq |h|\leq H$, is therefore,  
\[
\ll    \frac {1}{Q} \bigg/\left(\frac{ \eta M}{[q,q']}\right)^{9} \ll
    \frac {1}{Q}   \left(\frac{Q^2P}{ \eta x}\right)^{9} \ll x^{-\frac{20}{39}}  (x^{\frac{1}{39}+\sigma+\varepsilon})^{9}\ll x^{-1/10 },
\]
as  $\sigma \leq 1/75$ and $\varepsilon = \sigma / 3$,
which implies the result when we insert the formulas and estimates above into the sum for $g(q,q')$.

To establish \eqref{AlmostId}, let $c\pmod{d}$ be the residue class with $pc \equiv p'c \equiv a\pmod{d}$, so that $r \equiv c\pmod{d}$. Make a change of variables $r = ds+c$, so that
\[ \e_{d\ell\ell'}(rh) = \e_{\ell\ell'}(sh) \e_{d\ell\ell'}(ch)= \e_{\ell\ell'}(sh) +O(ch/d\ell\ell'). \]
The value of $s$ is determined by the congruence conditions
\[ p(ds+c) \equiv a\pmod{q}, \ \ p'(ds+c) \equiv a\pmod{q'}, \]
which can be rewritten as
\[ ps \equiv b \pmod{\ell}, \ \ p's \equiv b' \pmod{\ell'}, \]
where $b = (a-pc)/d$ and $b' = (a-p'c)/d$.
Since $\ell$ and $\ell'$ are coprime, the Chinese remainder theorem leads to
\[ s \equiv b((p\ell')^{-1}\operatorname{mod }\ell) \cdot \ell' + b'((p'\ell)^{-1}\operatorname{mod }\ell')\cdot \ell \pmod{\ell\ell'}.  \]
Hence
\[ \e_{\ell\ell'}(sh) = \e_\ell(bh \cdot \overline{p\ell'}) \e_{\ell'}(b'h \cdot \overline{p'\ell}). \]
Now apply the reciprocity relation
\[ \frac{v^{-1}\pmod{u}}{u} + \frac{u^{-1}\pmod{v}}{v} \equiv \frac{1}{uv}\pmod{1} \]
with $u=p\ell'$ and $v = p'\ell$ to obtain
\[ \e_{p\ell'}(w\overline{p'\ell}) \e_{p'\ell}(w\overline{p\ell'}) = \e_{pp'\ell\ell'}(w) \]
for any $w$. This implies that
\[ 
\e_{\ell'}(b'h \cdot \overline{p'\ell}) = \e_{p\ell'}(pb'h \cdot  \overline{p'\ell}) 
= \e_{p'\ell}(-b'h \cdot \overline{\ell'}) \e_{p'\ell\ell'}(b'h)
= \e_{p'\ell}(-b'h \cdot \overline{\ell'}) +O(b'h/p'\ell\ell') .
\]
The main term is therefore
\[
   \e_\ell(bh \cdot \overline{p\ell'})  \e_{p'\ell}(-b'h \cdot \overline{\ell'})=  \e_{p'\ell}(kh \cdot \overline{p\ell'}) ,
\]
as claimed.  For the error terms we note that $|c|\leq d$ and $|b'|\leq |a/d|+p'$; and therefore the error is
\[
\ll \left(1+ \frac{|a|}{dP}\right)  \frac {|h|}{\ell\ell'} \ll  \left(P^2+ |a|\right)  \frac {|h|}{Q^2} \ll \frac {|h|}{Q}
\]
as $|a|, x^{2\sigma}\leq Q$ and $|d|\leq P\leq x^\sigma$,
as claimed.
This completes the proof.
\end{proof}

Now we perform the summation over $q,q'$:

\begin{lemma}\label{lem:sumqq'} Suppose that $1\leq |a|\ll  Q\leq x^{\frac{20}{39}-\delta}$   and $P\leq x^{\varepsilon}$.
For any sequences  $\{\gamma(q)\}, \{\gamma'(q)\}$  with $|\gamma(q)| , |\gamma'(q)| \leq 1$  we have
\[ \sum_{\substack{q,q' \sim Q \\ (q,q') \mid p-p'}} \gamma(q) \gamma'(q') g(q,q') \ll  x^{2\sigma+\varepsilon}   Q^{2-\frac {1}{20}} +Q^2 x^{-1/10}
\ll \frac x{P^2}, \] 
assuming $\sigma =  \min\{ \delta/2,1/75\}$ and $\varepsilon=\sigma/3$.
\end{lemma}

\begin{proof}
The total contribution of the  error term  from Lemma~\ref{lem:g(q,q')} is   $O(Q^2 x^{-1/10})$.
For each integer $h,\ 1 \leq |h| \leq H:=x^{2\sigma}Q^2/dM$, each divisor $d$ of $p'-p$, and each fixed $|u|\leq  2d/Q$, we restrict our attention to 
those pairs with $(q,q') = d$. Write $q=d\ell$ and $q'=d\ell'$, and then define
\[ \alpha(n) = \begin{cases} \gamma(\ell d)\psi(\ell u) & \text{if }n = p'\ell  \text{ for some } \ell \sim Q/d\\ 0 & \text{otherwise},\end{cases} \]
so that   $|\alpha(n)|\leq |\psi(\ell u) |\ll 1$, and 
\[
\beta(n') = \begin{cases} (Q/\ell'd) \gamma'(\ell'd)\e_{\ell'd}(-Mhu) & 
 \text{if }n' = p\ell'  \text{ for some } \ell' \sim Q/d\\ 0 & \text{otherwise}.
\end{cases} 
\]

Theorem 1 of \cite{BC} (with $A=1,\ \theta=kh$ and the roles of $\alpha$ and $\beta$ swapped)  implies that 
\[ 
\sum_{\substack{n\sim N,\ n'\sim N' \\ (n,n')=1}} \alpha(n) \beta(n') \e_n(kh \cdot \overline{n'}) \ll 
\frac Qd \left( \frac{PQ}d \right)^{\frac {19}{20}+\frac{\varepsilon}{20}}  
\]
where $N=p'Q/d$ and $N'=pQ/d$, and  $k = (p'-p)a/d$, as
$|kh|/ (PQ/d)^2 \ll   |a|/  x^{1-2\sigma}\ll 1$.
 Summing over $u$ and $h$,  our sum is therefore
\[
\ll   \frac{M}{Q}  \sum_{d|p-p'}  \frac{2d}{Q} \cdot \frac{x^{2\sigma}Q^2}{dM} \cdot \frac Qd \left( \frac{PQ}d \right)^{\frac {19}{20}+\frac{\varepsilon}{20}}
\ll     Q x^{2\sigma}   \left( PQ  \right)^{\frac {19}{20}+\frac{\varepsilon}{20}}\ll x^{2\sigma+\varepsilon}   Q^{2-\frac {1}{20}}.
\]
\end{proof}

\begin{proof} [Completion of the  proof of Proposition~\ref{prop:bilinear-beyond1/2}] Each of $E_1,E_2,E_3$ can be effectively bounded using Lemma~\ref{lem:sumqq'}. Indeed, for $E_1$ we apply Lemma~\ref{lem:sumqq'} with 
\[ \gamma(q) = \begin{cases} \xi_q & \text{if }q \sim Q \\ 0 & \text{otherwise,} \end{cases} \ \  \gamma'(q') = \begin{cases} \xi_{q'} & \text{if }q' \sim Q \\ 0 & \text{otherwise.} \end{cases} \]
To bound $E_2$ we apply Lemma~\ref{lem:sumqq'} with
\[ \gamma(q_s) = \sum_{\substack{q_r \sim Q/q_s  }} \xi_{q_rq_s} q_r^{-1}, \ \  \gamma'(q') = \begin{cases} \xi_{q'} & \text{if }q' \sim Q \\ 0 & \text{otherwise.} \end{cases} \]
Similarly for $E_3$.
\end{proof}

\subsection{Proof of Theorems~\ref{Cor:Result2+} and~\ref{MathResult3+}} 
 
For $a$ fixed with
$1\leq |a|\ll  Q\leq x^{\frac{20}{39}-\delta}$ , we replace Lemma \ref{BilinBound} by Proposition~\ref{prop:bilinear-beyond1/2}  to obtain analogues of Theorem~\ref{MathResult2} and Corollary~\ref{MathResult2Cor}. The proof of Theorems~\ref{Cor:Result2+} and~\ref{MathResult3+} then follows in exactly the same way as the arguments in section~\ref{sec:deduce-many-thms}.

\section{Various examples}

\subsection{The large prime obstruction; the proof of Proposition \ref{LargePrimes2}}  
In this section we will construct multiplicative functions with certain properties.  We will only define these multiplicative functions on $[1,x]$ since the integers $n>x$ do not contribute. Let $\mathcal P$ be the set of primes $p \in (x/2,x]$ for which there does not exist an integer $q\sim Q$ with $q|p-1$.
Let $f_+,f_-$ be the multiplicative functions for which $f_+(n)=f_-(n)=g(n)$ for all $n\leq x$ with $n\not\in \mathcal P$, and $f_+(p)=1$ and $f_-(p)=-1$ for all $p\in \mathcal P$.
If $f \in \{f_+,f_-\}$ then
$$
\sum_{\substack{ n\leq x \\ n\equiv a \pmod q}} f(n) = \sum_{\substack{ n\leq x \\ n\equiv a \pmod q}} g(n) + \sum_{\substack{ x/2<p\leq x \\ p\equiv a \pmod q}} ( f(p)-g(p))  ,
$$
and so
$$
\Delta(f,x;q,1)-\Delta(g,x;q,1) =  \sum_{\substack{ x/2<p\leq x \\ p\equiv 1 \pmod q}}  ( f(p)-g(p))   - 
 \frac 1{\varphi(q)} \sum_{\substack{ x/2<p\leq x }}  ( f(p)-g(p)).
$$
Since   $f(p)=g(p)$ if $p\not \in \mathcal P$,  the first term on the right-hand side of the last displayed equation vanishes, by the definition of $\mathcal P$. Therefore
$$
\Delta(f_+,x;q,1)-\Delta(f_-,x;q,1) =  \frac {2\# \mathcal P} {\varphi(q)} \sim \frac {1} {\varphi(q)} \frac x{\log x}.
$$
The asymptotic in the last step follows from Theorem 6 and Corollary 2 of Ford's masterpiece  \cite{For} 
which imply that almost all primes $p$ in the interval $x/2<p\leq x$ belong to $\mathcal P$. Thus for each $q$ we have either $|\Delta(f_+,x;q,1)|$ or $|\Delta(f_-,x;q,1)|$ is $\gg \pi(x)/\varphi(q)$. The conclusion follows immediately.

\subsection{$f$ satisfying the Siegel-Walfisz criterion but not Bombieri-Vinogradov with prime moduli} \label{NoBV}   We will assume \eqref{ExtendPNT} in what follows. Let $x^{1/3}<Q<x^{1/2-\varepsilon}$, and let $\mathcal P$ now be the set of primes $p \in (x/2, x]$ for which there   exists a prime $q\sim Q$ for which $q|p-1$.
 We now define the multiplicative function $f$ so that $f(p)=-1$ if $p\in \mathcal P$, and $f(n)=1$ if $n \leq x$ and $n\notin\mathcal P$. 
 Therefore if $m<x/2$ and $(a,m)=1$ then
 \[
 \Delta(f,x;m,a) = \Delta(1,x;m,a) - 2 \Delta(1_{\mathcal P},x;m,a) = - 2 \Delta(1_{\mathcal P},x;m,a) +O(1).
\]
Note that if $p\in \mathcal P$ then $p-1$ can have at most two prime factors from the set of primes $q\sim Q$. Therefore, by inclusion-exclusion we have
 \[
 \# \mathcal P =  \sum_{\substack{q\sim Q \\ q \ \text{prime}}} \pi^*(x;q,1) - 
  \sum_{\substack{q_1<q_2\sim Q \\ q_1,q_2 \ \text{prime}}}   \pi^*(x;q_1q_2,1)
 \]
 where $\pi^*(x;q,a):=\pi(x;q,a)-\pi(x/2;q,a)$. Similarly if $m<Q$ (so that $(m,q)=1$, and therefore $\varphi(mq)=\varphi(m)\varphi(q)$),  then
 \[
 \#\{ p\in \mathcal P: \ p\equiv a \pmod m\} = \sum_{\substack{q\sim Q \\ q \ \text{prime}}}  \pi^*(x;qm,a_{qm}) -  \sum_{\substack{q_1<q_2\sim Q \\ q_1,q_2 \ \text{prime}}}   \pi^*(x;q_1q_2m,a_{q_1q_2m}) ,
 \]
 where $a_{rm}$ is that residue class mod $rm$ which is $\equiv 1 \pmod r$ and $\equiv a \pmod m$ (for $r=q$ or $q_1q_2$).
 Comparing these two equations,  we have
 \[
 \begin{split}
\Delta(1_{\mathcal P},x;m,a)  &= \sum_{\substack{q\sim Q \\ q \ \text{prime}}}  \left( \pi^*(x;qm,a_{qm})-\frac{\pi^*(x;q,1)}{\varphi(m)}  \right) \\
& \hskip1in  -  \sum_{\substack{q_1<q_2\sim Q \\ q_1,q_2 \ \text{prime}}}   \left( \pi^*(x;q_1q_2m,a_{q_1q_2m}) -\frac{\pi^*(x;q_1q_2,1)}{\varphi(m)} \right)  \\
& \ll  \frac{\pi(x)}{(\log x)^A} \sum_{\substack{q\sim Q \\ q \ \text{prime}}}  \frac{1}{\varphi(mq)}  + \frac{\pi(x)}{(\log x)^A}   \sum_{\substack{q_1<q_2\sim Q \\ q_1,q_2 \ \text{prime}}}   \frac{1}{\varphi(mq_1q_2)} \\ & \ll \frac 1{\varphi(m)} \frac{x}{(\log x)^{A+2}},
\end{split}
 \]
 using \eqref{ExtendPNT} under the assumption  $4mQ^2\leq x^{1 -\varepsilon}$. This shows that the $A$-Siegel-Walfisz criterion for $f$ holds for every $A>0$,  for the sums over $n\leq x$. The same argument works for sums over $n\leq X$, for any $X\leq x$, changing the definition of $\pi^*(X;q,a)$ to be $0$ if $X\leq x/2$, and 
 $\pi^*(X;q,a):=\pi(X;q,a)-\pi(x/2;q,a)$ if $x/2<X\leq x$.

 On the other hand, \eqref{ExtendPNT} (or the unconditional Brun-Titchmarsh Theorem), implies  that 
\[
 \# \mathcal P \ll \sum_{\substack{q\sim Q \\ q \ \text{prime}}} \frac{\pi^*(x)}{\varphi(q)} \ll \frac x{(\log x)^2} .
 \]
Therefore if $q\sim Q$ and $q$ is prime then  
\[
\Delta(1_{\mathcal P},x;q,1) = \pi^*(x;q,1)- \frac{\# \mathcal P} {\varphi(q)} \sim \frac{\pi^*(x)}{\varphi(q)}.
\]
Therefore if $q\sim Q$ and $q$ is prime then
\[
 \Delta(f,x;q,1) \sim - \frac{1}{q} \cdot \frac x{\log x},\]
and so
\[
\sum_{q\sim Q}  |\Delta(f,x;q,1)| \geq  \sum_{\substack{q\sim Q \\ q \ \text{prime}}}  |\Delta(f,x;q,1)|  \gg \frac x{(\log x)^2} .
\]
This doesn't quite prove Corollary~\ref{LargePrimes2b}, but it will be fixed in the next subsection.

\subsection{$f$ satisfying the Siegel-Walfisz criterion but not Bombieri-Vinogradov; the proof of Corollary  \ref{LargePrimes2b} } \label{NoBV2}
Now we use a minor modification of the argument in the previous subsection to extend the estimate to all moduli.  Let
$x^{2/5}<Q<x/2$, and $I=(x^{1/3},x^{2/5}]$. Let $\mathcal P$ now be the set of primes $p \in (x/2, x]$ for which there   exists a prime $\ell \in I$ with $\ell |p-1$. Define the multiplicative function $f$ so that $f(p)=-1$ if $p\in \mathcal P$, and $f(n) = 1$ if $n \leq x$ and $n\notin\mathcal P$. 
So we again have that if $m<x/2$ and $(a,m)=1$ then
 \[ \Delta(f,x;m,a) = - 2 \Delta(1_{\mathcal P},x;m,a) +O(1). \]
 The  argument of Section~\ref{NoBV} yields  the Siegel-Walfisz criterion for $f$, assuming \eqref{ExtendPNT}.
 We also have, assuming \eqref{ExtendPNT}, that 
 \[
  \# \mathcal P \leq \sum_{\substack{ \ell\in I \\ \ell \ \text{prime}}} \pi^*(x;\ell,1) = \{ 1+o(1)\}\sum_{\substack{ \ell\in I \\ \ell \ \text{prime}}}  \frac{\pi^*(x)}{ \ell-1}   = \{ c+o(1)\} \pi^*(x) ,
 \]
 where $c=\log(6/5)<1/5$.

Now suppose that $q\sim Q$ (not necessarily prime) and that there exists a prime $\ell$ in the interval $I$ which divides $q$. If $x/2<p\leq x$ and $p\equiv 1 \pmod q$ then $\ell|q|p-1$ and so $p\in \mathcal P$. So we have
\[
\Delta(1_{\mathcal P},x;q,1) = \pi^*(x;q,1)- \frac{\# \mathcal P} {\varphi(q)}  \geq (1-c+o(1)) \frac{\pi^*(x)}{\varphi(q)} \geq \frac 45 \cdot \frac{\pi^*(x)}{\varphi(q)},
\]
and therefore 
\[
- \Delta(f,x;q,1) \gg   \frac{1}{\varphi(q)} \cdot \frac x{\log x}.
\]
Now each such $q$ can have at most one such prime divisor $\ell$. Therefore
\[
\sum_{q\in Q}  |\Delta(f,x;q,1)| 
\gg  \frac{x}{\log x} \sum_{\substack{\ell \in I \\ \ell\  \text{prime}}} \sum_{\substack{q \sim Q \\ \ell | q}} \frac{1}{\varphi(q)} \gg
\frac x{\log x} \sum_{\substack{ \ell\in I \\ \ell \ \text{prime}}} \frac 1 \ell \gg  \frac x{\log x}.
\]
This completes the proof of Corollary~\ref{LargePrimes2b}.

\subsection{$f$ that correlate to many characters} \label{ManyChar}

For a given integer $q$ suppose we are given constants $g(\chi)$ for each character $\chi \pmod q$, and define the multiplicative function $f$ by
\[
\Lambda_f(n) = c_q(n) \Lambda(n) , \ \text{where} \ c_q(n):= \sum_{\chi \pmod q} g(\chi) \chi(n)
\]
depends only on $n\pmod q$. This 
  implies that the Dirichlet series $F$ associated to $f$ is
\[
F(s) = \prod_{\chi \pmod q}  L(s,\chi)^{g(\chi)}.
\]
Moreover if we twist $f$ by a character $\psi\pmod{q} $, then
\[
F_\psi(s):= \sum_{n\geq 1} \frac{(f\overline\psi)(n) }{n^s}  = \prod_{\chi \pmod q}  L(s,\chi)^{g(\chi\psi)}.
\]
By the Selberg-Delange Theorem (see, e.g., Theorem 5 of Section II.5.5 in~\cite{Ten}) one has that 
if $g(\psi)$ is not an integer then
\[
|S_f(x,\psi)| \gg  \frac x{ (\log x)^{1-\text{Re}(g(\psi))}} \ .
 \]
 
 Now $f\in \mathcal C$ if and only if $|c_q(n)| \leq 1$ for all $n$. We will make the particular choice that 
 \[
 g(\chi) = \frac 1{\varphi(q)} \sum_{a \pmod q} \overline{\chi}(a) e(a/q)
 \]
 are Gauss sums (where $e(t):=e^{2i\pi t}$), so that
 \[
 \begin{split}
 c_q(n)&= \frac 1{\varphi(q)}\sum_{\chi \pmod q} \chi(n) \sum_{a \pmod q} \overline{\chi}(a) e(a/q) \\
 & = \sum_{a \pmod q}  e(a/q) \frac 1{\varphi(q)} \sum_{\chi \pmod q} \chi(n)\overline{\chi}(a) = e(n/q) .
 \end{split}
  \]
 Therefore $f\in \mathcal C$ and $| g(\chi) |=\sqrt{q}/(q-1)$ whenever $\chi$ is non-principal and $q$ is prime.
 Moreover Katz showed (though see \cite{KZ} for a much easier proof) that the arguments of the $g(\chi)$ are equi-distributed around the unit circle.
 In particular the number of  non-principal $\psi \pmod q$ for which $1-\text{Re}(g(\psi))>1-\frac 1{2\sqrt{q}}$ is $\sim q/3$. We deduce that if $q$ is a prime for which $3k\lesssim q<4/\epsilon^2$ then there are $>k$ primitive characters $\psi \pmod q$ for which 
 \[
|S_f(x,\psi)| \gg  \frac x{ (\log x)^{1-\epsilon}} \ .
 \]
 Therefore, by \eqref{eq: 3.lbs} and the prime number theorem, we deduce that 
for  any given integer $k$ there is an  
$ \varepsilon' \sim \frac 2{\sqrt{3k}}$, for which there exists $f\in {\mathcal C} $ such that 
\[ 
\sum_{q\sim Q} \max_{a:\ (a,q)=1} | \Delta_{k}(f,x;q,a) | \gg   \frac{x} {(\log x)^{1-\varepsilon'}} .
\]
This shows that  Theorem \ref{Cor:Result2} is close to ``best possible".

\section{Average of higher $U^k$-norms}\label{sec:higher-uk}

In this section, we investigate a higher order generalization of the Bombieri-Vinogradov inequality, which measures more refined distributional properties. This higher order version involves Gowers norms, a central tool in additive combinatorics. We refer the readers to~\cite[Chapter 11]{TV06} for the basic definitions and applications. In particular, $\|f\|_{U^k(Y)}$ stands for the $U^k$-norm of the function $f$ on the discrete interval $[0,Y]\cap \Z$.

For any arithmetic function $f: \Z \to \C$ and any residue class $a \pmod q$, denote by $f(q\cdot + a)$ the function $m \mapsto f(qm+a)$. 

\begin{corollary}\label{cor:uk-average}
Fix a positive integer $k$. Fix $B \geq 2$ and $\ee > 0$.
Let $2 \leq y \leq x^{1/10}$ be large. Let $f$ be a completely multiplicative function with each $|f(n)| \leq 1$, which is only supported on $y$-smooth integers. Assume that
\begin{equation}\label{eq:savelogpower} 
\sum_{\substack{n \leq x \\ n \equiv a\pmod{q}}} f(n) \ll_C \frac{x}{(\log x)^C}
\end{equation}
for any $(a,q)=1$ and any $C \geq 1$. Let $Q = x^{1/2}/(y^{1/2}(\log x)^A)$, where $A$ is sufficiently large in terms of $k,B,\ee$. Then for all but at most $Q(\log x)^{-B}$ moduli $q \leq Q$, we have
\[ \max_{\substack{0 \leq a < q \\ (a,q)=1}} \|f(q \cdot +a)\|_{U^k(x/q)} \leq \ee. \] 
\end{corollary}
 
\begin{proof}
In view of Corollary~\ref{MathResult3}, the hypothesis on $f$ implies that
\begin{equation}\label{eq:BV-f} 
\sum_{q \leq Q} \max_{(a,q)=1} \left|\sum_{\substack{n \leq x \\ n \equiv a\pmod{q}}} f(n) \right| \ll \frac{x}{(\log x)^{B'}}, 
\end{equation}
for any constant $B' = B'(B)$, provided that $A$ is chosen large enough. The conclusion of Corollary~\ref{cor:uk-average} follows by using this estimate in place of the Bombieri-Vinogradov inequality for the Liouville function $\lambda$ in the argument in~\cite[Section 2.2]{Sha}.
 
More precisely, if $q \leq Q$ is an exceptional moduli in the sense that
\[ \|f(q \cdot +a_q)\|_{U^k(x/q)} \geq \ee \]
for some residue class $a_q\pmod{q}$, where $0 \leq a_q < q$ and $(a_q,q)=1$, then by the inverse theorem for Gowers norms~\cite{GTZ} $f$ must correlate with a nilsequence of complexity $O_{\ee}(1)$ on the progression $\{n \leq x : n \equiv a_q\pmod{q}\}$. Now apply~\cite[Lemma 2.4]{Sha} and follow the deduction of~\cite[Theorem 1.2]{Sha}, with $f$ in place of $\lambda$. The number of exceptional moduli satisfying~\cite[Equation (2.3)]{Sha} can be satisfactorily bounded using~\eqref{eq:BV-f}, and the rest of the argument already applies to a general multiplicative function $f$.
\end{proof}

The hypothesis \eqref{eq:savelogpower} of Corollary \ref{cor:uk-average} follows from assuming the Siegel-Walfisz criterion, along with  
\[ 
\sum_{\substack{n \leq x  }} f(n) \ll_C \frac{x}{(\log x)^C}
\]
uniformly in $x$.

Corollary \ref{cor:uk-average} only applies to smoothly supported $f$ since  we need to save  an arbitrary power of $\log x$ in the Bombieri-Vinogradov inequality for $f$  for the results in~\cite{Sha} to be applicable. However we would guess that our estimate
\[ \max_{\substack{0 \leq a < q \\ (a,q)=1}} \|f(q \cdot + a)\|_{U^k(x/q)} = o(1) \]
holds on average over $q \leq x^{1/2-o(1)}$, for any $1$-bounded multiplicative function $f$ satisfying the $1$-Siegel-Walfisz criterion.


\end{document}